\documentclass[final,leqno,onefignum,onetabnum]{siamltex}

\usepackage{setspace,graphicx}
\usepackage[english]{babel}
\usepackage{lmodern}
\usepackage{indentfirst}
\usepackage{placeins, subfigure}

\usepackage{lettrine}
\usepackage{booktabs}
\usepackage{amsmath,amssymb}
\usepackage{braket}
\usepackage{chngpage}
\usepackage{cases}
\usepackage{graphicx}
\usepackage{tabularx}
\usepackage{multirow}

\newtheorem{remark}{Remark}[section]
\newtheorem{assumption}{Assumption}[section]

\usepackage{changes}
\usepackage{todonotes}

\newcommand{\seminorm}[1]{\left|#1 \right|}

\renewcommand{\emph}[1]{{\it #1}}

\def\Vh{{W}_h}
\def\VE{W_{h|E}}
\def\VEt{\widetilde{V}_{h|E}}
\def\P{{\mathbb P}}
\def\Vhbc{{W}_h^0}

\usepackage{setspace}

\title{A $C^1$ virtual element method for the Cahn-Hilliard equation with polygonal meshes}

\author{Paola F. Antonietti\thanks{MOX-Dipartimento di Matematica, Politecnico di Milano, E-mail: paola.antonietti@polimi.it} \and Lourenco Beir\~ao da Veiga\thanks{Dipartimento di Matematica,  
Universit\`a degli Studi di Milano, E-mail: lourenco.beirao@unimi.it} \and Simone Scacchi
\thanks{Dipartimento di Matematica,  Universit\`a degli Studi di Milano, E-mail: simone.scacchi@unimi.it}
\and Marco Verani \thanks{MOX-Dipartimento di Matematica, Politecnico di Milano, E-mail: marco.verani@polimi.it}}

\begin{document}

\maketitle

\begin{abstract}
In this paper we develop an evolution of the $C^1$ virtual elements of minimal degree for the approximation of the Cahn-Hilliard equation. 
The proposed method has the advantage of being conforming in $H^2$ and making use of a very simple set of degrees of freedom, namely 3 degrees of freedom per vertex of the mesh. 
Moreover, although the present method is new also on triangles, it can make use of general polygonal meshes.
As a theoretical and practical support, we prove the convergence of the semi-discrete scheme and investigate the performance of the fully discrete scheme through a set of numerical tests. 
\end{abstract}

\begin{keywords} Virtual element method, Cahn-Hilliard
\end{keywords}

\begin{AMS} 
65M99 
\end{AMS}

\pagestyle{myheadings}
\thispagestyle{plain}
\markboth{}{}

\section{Introduction}
\label{sec:intro}

The study of the evolution of transition interfaces, which is of paramount importance in many physical/biological phenomena and industrial processes, can be grouped into two macro classes, each one corresponding to a different method of dealing with the moving free-boundary: the sharp interface method and the phase-field method. In the sharp interface approach, the free boundary is to be determined together with the solution of suitable partial differential equations where proper jump relations have to imposed across the free boundary. In the phase field approach, the interface is specified as the level set of a smooth continuos function exhibiting large gradients across the interface.   

Phase field models, which date back to the works of Korteweg \cite{Korteweg}, Cahn and Hilliard \cite{Cahn:1961,Cahn-Hilliard:1958, Cahn-Hilliard:1959}, Landau and Ginzburg \cite{Landau-Ginzburg:1965} and van der Waals \cite{Waals:1979}, have been classicaly employed to describe phase separation in binary alloys. However, recently Cahn-Hilliard type equations have been extensively used in an impressive variety of applied problems, such as, among the others, tumor growth \cite{Wise-Lowengrub:2008, Oden-tumor:2010}, origin of Saturn's rings \cite{saturn}, separation of di-block copolymers \cite{copolymers}, population dynamics \cite{population}, image processing \cite{Bertozzi:2007b} and even clustering of mussels \cite{mussels}.  

Due to the wide spectrum of applications, the study of efficient numerical methods for the approximate solution of the Cahn-Hilliard equation has been the object of an intensive research activity.  
Summarizing the achievements in this field is a tremendous task that go beyond the scope of this paper. 
Here, we limit ourvselves to some remarks on finite element based methods, as the main properties (and limitations) of these schemes are instrumental to motivate the introduction of our new approach. 
As the Cahn-Hilliard equation is a fourth order nonlinear problem, a natural approach is the use of 
 $C^1$ finite elements (FEM) as in \cite{Elliott-Zheng:1986, Elliott-French:1987}. 
However, in order to avoid the well known difficulty met in the implementation of $C^1$ finite elements, another possibility is the use of non-conforming (see, e.g., \cite{Elliott-French:1989}) or discontinuous (see, e.g., \cite{Wells-Kuhl-Garikipati:2006}) methods; the drawback is that in such case the discrete solution will not satisfy a $C^1$ regularity. 
Alternatively, the most common strategy employed {\it in practice}  to solve the Cahn-Hilliard equation with (continuos and discontinuous) finite elements is to use mixed methods (see e.g. \cite{Elliott-French-Milner:1989, Elliott-Larsson:1992} and \cite{Kay-Styles-Suli:2009} for the continuous and discontinuous setting, respectively).  Clearly, the drawback of this approach is the increase of the numbers of degrees of freedom, and thus of the computational cost. Very recently, the difficulty related to the practical use of $C^1$ basis functions has been addressed with success also in the framework of isogeometric analysis \cite{Gomez-Hughes:2008}. 

In this paper, we introduce and analyze the $C^1$ virtual element method (VEM)  for the approximate solution of the Cahn-Hilliard equation. This newly introduced method (see, e.g., \cite{volley} for an introduction to the method and \cite{hitchhikers} for the details of its practical implementation) is characterized by the capability of dealing with very general polygonal/polyedral meshes and to possibility of easily implementing highly regular discrete spaces. Indeed, by avoiding the explicit construction of the local basis functions, the VEM can easily handle general polygons/polyhedrons without complex integrations on the element. In addition, thanks to this added flexibility, it was discovered \cite{kirchhoff,arbitrary} that virtual elements can also be used to build global discrete spaces of arbitrary regularity ($C^1$ and more) that are quite simple in terms of degrees of freedom and coding. Other virtual element contributions are, for instance \cite{Brezzi-Falk-Marini,ABMV14,VEM-elasticity,Berrone-VEM,hourglass,Paulino-VEM,VEM19,VemSteklov}, while for a very short sample of other FEM-inspired methods dealing with general polygons we refer to 
\cite{Bishop,Cockburn-IMU,DiPietro-Ern-1,Droniou-gradient,Gillette-2,TPPM10,Wachspress11,Wang-1}.

In the present contribution we develop a modification of the $C^1$ virtual elements (of minimal degree) of \cite{arbitrary}  for the approximation of the Cahn-Hilliard equation. Also taking inspiration from the enhancement techniques of \cite{projectors}, we define the virtual space in order to be able to compute three different projection operators, that are used for the construction of the discrete scheme. Afterwards, we prove the convergence of the semi-discrete scheme and investigate the performance of the fully discrete scheme numerically. We underline that, on our knowledge, this is the first application of the newborn virtual element technology to a nonlinear problem.\\

The paper is organized as follows. In Section~\ref{sec:2} we describe the proposed virtual element method. In Section
\ref{sec:error} we develop the theoretical error estimates. In Section~\ref{sec:num} we present the numerical tests.

\section{The continuous and discrete problems}
\label{sec:2}

In this section, after presenting the Cahn-Hilliard equation, we introduce the Virtual Element discretization.
The proposed strategy takes the steps from the $C^1$ methods described in \cite{kirchhoff,arbitrary} for the Kirchhoff and Poisson problems, respectively, combined with an enhancement strategy first introduced in \cite{projectors}. The present virtual scheme makes use of three different projectors and of a particular construction to take care of the nonlinear part of the problem.

\subsection{The continuous problem}
\label{sec:cont}
Let $\Omega\subset \mathbb{R}^2$ be an open bounded domain. Let $\psi(x)= (1-x^2)^2/4$ and let $\phi(x) = \psi^\prime (x)$,  we consider the following Cahn-Hilliard problem: 
find $u(x,t):\Omega \times [0,T] \rightarrow {\mathbb R}$ such that:
\begin{equation}\label{eq:CHforte}
\left\{
\begin{aligned}
& \partial_t u - \Delta \big( \phi(u) - \gamma^2 \Delta u(t) \big) = 0 
&& \textrm{ in } \Omega \times [0,T], \\
& u(\cdot, 0) = u_0(\cdot) 
&& \textrm{ in } \Omega, \\
& \partial_{\bf n} u = \partial_n \big( \phi(u) - \gamma^2 \Delta u(t) \big) = 0
&& \textrm{ on } \partial\Omega \times [0, T],
\end{aligned}
\right.
\end{equation}
where $\partial_{\bf n}$ denotes the (outward) normal derivative and $\gamma \in {\mathbb R}^+$, $0 < \gamma \ll 1$, represents the interface parameter. Throughout the paper we will employ the standard notation for Sobolev spaces \cite{Adams}.  We now introduce the variational form of \eqref{eq:CHforte} that will be used to derive the virtual element discretization. To this aim, we preliminary define the following bilinear forms 
\begin{equation*}
\begin{aligned}
& a^\Delta(v,w) = \int_\Omega (\nabla^2 v) : (\nabla^2 w) \, {\rm d}x 
&& \forall v,w \in H^2(\Omega), \\
& a^\nabla(v,w) = \int_\Omega (\nabla v) \cdot (\nabla w) \, {\rm d}x 
&& \forall v,w \in H^1(\Omega), \\
& a^0(v,w) = \int_\Omega v \, w \, {\rm d}x 
&& \forall v,w \in L^2(\Omega),
\end{aligned}
\end{equation*}
and the semi-linear form 
\begin{equation*}
r(z;v,w) = \int_\Omega \phi'(z) \nabla v \cdot \nabla w \, {\rm d}x 
\quad \forall z,v,w \in H^2(\Omega)
\end{equation*}
where all the symbols above follow a standard notation.
Finally, introducing the space
\begin{equation}\label{V}
V = \big\{ v \in H^2(\Omega) \: : \: \partial_{\bf n} u = 0 \textrm{ on } \partial \Omega \big\}
\end{equation}
the weak formulation of problem \eqref{eq:CHforte} reads as: find $u(\cdot,t) \in V$ such that
\begin{equation}\label{contpbl}
\left\{
\begin{aligned}
& a^0(\partial_t u,v) + { \gamma^2} a^\Delta(u,v) + r(u;u,v) = 0 \quad \forall v \in V ,   \\
& u(\cdot,0)=u_0(\cdot) . 
\end{aligned}
\right.
\end{equation}
In the theoretical analysis of Section \ref{sec:error}, we will work under the following regularity assumption on the solution of \eqref{contpbl}
\begin{equation}\label{CH-regularity}
u\in C^1(0,T;H^4(\Omega)\cap V),
\end{equation}
see, e.g., \cite{Temam:1989} for a possible proof under higher regularity hypotheses on the initial datum $u_0$.

\subsection{A $C^1$ Virtual Element space}
\label{sec:vem1}
In the present section we describe the virtual element space $\Vh\subset H^2(\Omega)$ that we will use in the next section to build a discretization of problem \eqref{contpbl}. From now on, we will assume that $\Omega$ is a polygonal domain in $\mathbb{R}^2$. \\

Our construction will need a few steps. Let $\Omega_h$ represent a decomposition of $\Omega$ into general polygonal elements $E$ of diameter $h_E$.
In the following, we will denote by $e$ the edges of the mesh $\Omega_h$ and, for all $e \in \partial E$, ${\bf n}_E^e$ will denote the unit normal vector to $e$ pointing outward to $E$. We will use the symbol $\P_k(\omega)$ to denote the space of polynomials of degree less than or equal to $k$ living on the set $\omega \subseteq {\mathbb R}^2$.
Finally, we will make use of the following local bilinear forms for all $E\in\Omega_h$
\begin{equation}\label{loc-forms-cont}
\begin{aligned}
a_E^\Delta(v,w) & = \int_E (\nabla^2 v) : (\nabla^2 w) \, {\rm d}x && \forall v,w \in H^2(E), \\
a_E^\nabla(v,w) & = \int_E (\nabla v) \cdot (\nabla w) \, {\rm d}x &&   \forall v,w \in H^1(E), \\
 a_E^0(v,w) &= \int_E v \, w \, {\rm d}x && \forall v,w \in L^2(E).
\end{aligned}
\end{equation}
Given an element $E\in\Omega_h$, the \emph{augmented} local space $\VEt$ is defined by
\begin{equation}\label{Vtilde}
\begin{aligned}
\VEt = \Big\{ v \in H^2(E) \: : \: \Delta^2 v \in \P_{2}(E) , \
& v_{|\partial E} \in C^0(\partial E) , v_{|e} \in \P_3(e) \ \forall e \in \partial E , \\
& \nabla v_{|\partial E} \in [C^0(\partial E)]^2 ,
\partial_{\bf n} v_{|e} \in \P_1(e) \ \forall e \in \partial E
\Big\} ,
\end{aligned}
\end{equation}
with $\partial_{\bf n}$ denoting the (outward) normal derivative.
The space $\VEt$ is made of functions that are continuous and piecewise cubic on the boundary, with continuous gradient on the boundary, normal linear component on each edge and such that its bilaplacian is a quadratic polynomial.

We now introduce two sets $D1$ and $D2$ of linear operators from $\VEt$ into $\mathbb{R}$. For all $v_h \in \VEt$ they are defined as follows:
\begin{itemize}
\item[$D1$] contains linear operators evaluating $v_h$ at the $n=n(E)$ vertexes of $E$;
\item[$D2$] contains linear operators evaluating $\nabla v_h$ at the $n=n(E)$ vertexes of $E$.
\end{itemize}
Note that, as a consequence of definition \eqref{Vtilde}, the output values of the two sets of operators $D1$ and $D2$ are sufficient to uniquely determine $v_h$ and $\nabla v_h$ on the boundary of $E$.

Let us now introduce  the projection operator $\Pi ^{\Delta}_E \colon \VEt \rightarrow \P_2(E)$, defined by
\begin{equation}\label{pi_delta}
\left\{
\begin{aligned}
& a_E^\Delta(\Pi^{\Delta}_E v_h, q) = a_E^\Delta(v_h,q) \quad \forall q \in \P_2(E) \\
& (\!( \Pi^{\Delta}_E v_h , q )\!)_E = (\!( v_h, q)\!)_E \quad \forall q \in \P_1(E) ,
\end{aligned}
\right.
\end{equation}
for all $v_h \in \VEt$ where $(\!( \cdot, \cdot )\!)_E$ represents an euclidean scalar product acting on the function vertex values, i.e. 
$$
(\!( v_h, w_h )\!)_E {=} \sum_{\stackrel{\nu \textrm{ vertexes}}{\textrm{of } \partial E}} 
\!\! v_h(\nu) \: w_h(\nu) \quad \forall v_h,w_h \in C^0(E) .
$$
Some explanation is in order to motivate the construction of the operator $\Pi^{\Delta}_E$. 
First, we note that the bilinear form $a_E^\Delta(\cdot,\cdot)$ has a non-trivial kernel, given by $\P_1(E)$. Hence, the role of the second condition in \eqref{pi_delta} is to select an element of the kernel of the operator. Moreover, it is easy to check that the operator $\Pi^{\Delta}_E$ is well defined, as for all $v_h \in \VEt$ it returns one (and only one) function $\Pi^{\Delta}_E v_h \in \P_2(E)$.
Second, it is crucial to remark that the operator $\Pi^{\Delta}_E$ is uniquely determined on the basis of the informations carried by  the linear operators in $D1$ and $D2$. Indeed, it is sufficient to perform a double integration by parts on the right hand side of \eqref{pi_delta}, which gives
$$
a_E^\Delta(v_h,q) = \int_E \nabla^2 v_h : \nabla^2 q {\rm d}x = 
\int_{\partial E}  \big( \nabla^2(q) {\bf n}_E^e \big) \cdot \nabla v_h {\rm d}s 
- \int_{\partial E}  v_h (\textrm{div} \nabla^2 q )\cdot {\bf n}_E^e ,
$$
and to observe that the above term on the right hand side only depends on the boundary values of $v_{h}$ and $\nabla v_h$. We note that the same holds for the right hand side of the second equation in  \eqref{pi_delta}, since it depends only on the vertex values of $v_h$.  
To conclude, as for any  $v_h \in \VEt$, the output values of the linear operators in $D1$ and $D2$ are sufficient to define $v_h$ and $\nabla v_h$ on the boundary, it turns out that the operator $\Pi^{\Delta}_E$ is uniquely determined on the basis of the evaluations performed by the linear operators in $D1$ and $D2$.

We are now ready to define our virtual local spaces
\begin{equation}\label{VE}
\VE = \big\{ v \in \VEt \: : \: \int_E \Pi^\Delta_E (v_h) \, q \,  {\rm d}x = \int_E v_h \, q \, {\rm d}x \quad \forall q \in \P_2 (E) \big\} .
\end{equation}
We observe that, since $\VE  \subset \VEt$, the operator $\Pi^{\Delta}_{E}$ is well defined on $\VE$ and computable only on the basis of the output values of the operators in $D1$ and $D2$.

Moreover, we have the following result.
\begin{lemma}
The set of operators $D1$ and $D2$ constitutes a set of degrees of freedom for the space $\VE$.
\end{lemma}
\begin{proof} 
We start by noting that the space $\VEt$ is associated to a well posed biharmonic problem on $E$ with Dirichlet boundary data and standard volume loading, i.e., 
$$
\left\{
\begin{aligned}
& - \Delta^2 v_h \textrm{ assigned in } E , \\
& v_{h} \textrm{ and } \partial_{\bf n} v_{h} \textrm{ assigned on } \partial E. 
\end{aligned}
\right.
$$ 
Thus the dimension of $\VEt$ equals the dimension of the data space (loading and boundary data spaces). 
We now recall that, as already noted, the operators $D1$ and $D2$ uniquely determine $v_h$ and $\nabla v_h$ on the boundary of $E$ and thus the cardinality
$
\# \{ D1  \}  + \# \{  D2 \} 
$
exactly corresponds to the dimension of the boundary data in the above biharmonic problem. Therefore, since the loading data space has dimension equal to $\textrm{dim}(\P_2(E))$, we have
$$
\textrm{dim}(\VEt) = \# \{ D1  \}  + \# \{  D2 \}  + \textrm{dim}(\P_2(E)) .
$$
Now, we observe that the space $\VE$ is a subspace of $\VEt$ obtained by enforcing the constraints in \eqref{VE}, i.e a set of $n$ linear equations, with $n= \text{dim}(\P_2(E))$. Since such equations could, in principle, not be linearly independent, all we can say on the dimension of $\VE$ is
\begin{equation}
\textrm{dim}(\VE) \ge \textrm{dim}(\VEt) - \textrm{dim}(\P_2(E))  = \# \{ D1  \}  + \# \{  D2 \}  .
\end{equation}
The proof is therefore complete if we show that any $v_h \in \VE$ that vanishes on $D1$ and $D2$ is indeed the zero element of $\VE$. Let $v_h \in \VE$ vanish on 
$D1$ and $D2$. First of all, this easily implies that $v_h$ and $\nabla v_h$ are null on the boundary $\partial E$. Moreover, since the operator $\Pi^{\Delta}_E$ is linear and depends only on the output values of the operators in $D1$ and $D2$, it must hold $\Pi^{\Delta}_E (v_h) = 0$. Recalling definition \eqref{VE}, this in turn yields
\begin{equation}\label{X1}
\int_E v_h \, q \, {\rm d}x = 0 \quad \forall q \in \P_2 (E) .
\end{equation}
Since $v_h \in \VE \subseteq \VEt$, we have $\Delta^2 v_h \in \P_2(E)$. Therefore, we can take $q=\Delta^2 v_h$ as a test function in \eqref{X1}. A double integration by parts, using also that $v_h$ and $\nabla v_h$ are null on $\partial E$, then gives
$$
0 = \int_E v_h \, \Delta^2 v_h \, {\rm d}x = \int_E  \Delta v_h \, \Delta v_h \, {\rm d}x .
$$
Thus $\Delta v_h =0$ and the proof is complete by recalling again the boundary conditions on $v_h$.
\end{proof}

The space $\VE$ satisfies also the following properties.
The first one is that 
$$
\P_2(E) \subseteq \VE ,
$$
that will guarantee the good approximation properties for the space. The above inclusion is easy to verify, since clearly $\P_2(E) \subseteq \VEt$ and the additional condition in \eqref{VE} is satisfied by $\P_2(E)$ polynomials (being $\Pi^\Delta_E$ a projection on such polynomial space).
The second property is that also the standard $L^2$ projection operator 
$\Pi^0_{E}:  \VE \rightarrow \P_2(E)$ is computable (only) on the basis of the values of the degrees of freedom $D1$ and $D2$. 
Indeed, for all $v_h\in\VE$, the function $\Pi^0_{E} v_h \in \P_2(E)$ is defined by
\begin{equation}\label{pi_zero}
a_E^0(\Pi^{0}_E v_h, q) = a_E^0(v_h,q) \quad \forall q \in \P_2(E) ,
\end{equation}
where the bilinear form $a_E^0(\cdot,\cdot)$ , introduced in \eqref{loc-forms-cont}, represents the $L^2(E)$ scalar product.
Due to the particular property appearing in definition \eqref{VE}, the right hand side in \eqref{pi_zero} is computable using $\Pi^{\Delta}_E v_h$, and thus $\Pi^{0}_E v_h$ depends only on the values of the degrees of freedom $D1$, $D2$ attained by $v_h$ and $\nabla v_h$.
Actually, it is easy to check that on the space $\VE$ the projectors  $\Pi^\Delta_E$ and $\Pi^{0}_E$ are the same operator (although for the sake of clarity we prefer to keep the notation different).

We introduce an additional projection operator that we will need in the following. We define $\Pi^\nabla_{E}:  \VE \rightarrow \P_2(E)$ by
\begin{equation}\label{pi_nabla}
\left\{
\begin{aligned}
& a_E^\nabla(\Pi^{\nabla}_E v_h, q) = a_E^\nabla(v_h,q) \quad \forall q \in \P_2(E) \\
& \int_E \Pi^{\nabla}_E v_h \, {\rm d}x = \int_E v_h \, {\rm d}x .
\end{aligned}
\right.
\end{equation}
We remark that, since the bilinear form $a_E^\nabla(\cdot,\cdot)$ has a non trivial kernel (given by the constant functions) we added a second condition in order to keep the operator $\Pi^{\nabla}_E$ well defined.
It is easy to check that the right hand side in \eqref{pi_nabla} is computable on the basis of the values of the degrees of freedom $D1$ and $D2$. For the first equation in \eqref{pi_nabla}, this can be shown with an integration by parts (similarly as already done for the $\Pi^{\Delta}_E$ projector) 
$$
\int_E \nabla v_h \cdot \nabla q {\rm d}x = 
- (\Delta q)_{|E} \int_{E} v_h {\rm d}x 
+ \int_{\partial E}  v_h \, \partial_{\bf n} q \, {\rm d} s
$$
and noting that the identity
$$
\int_E v_h {\rm d}x = \int_E \Pi^{0}_E  v_h {\rm d}x
$$
allows to compute the integral of $v_h$ on $E$ using only the values of the degrees of freedom $D1$ and $D2$. For the ease of the reader, we summarize what we have accomplished so far in the following remark.
\begin{remark}
We have introduced a set of local spaces $\VE$ (well defined on general polygons and containing $\P_2(E)$) and the associated local degrees of freedom. We have moreover shown that we have three different projection operators (each one associated to a different bilinear form appearing in the problem) that can be computed making use only of the values of such degrees of freedom. 
\end{remark}

The global discrete space can now be assembled in the classical finite element fashion, yielding
$$
\Vh = \big\{ v \in V \: : \: v|_E \in \VE \quad \forall E\in\Omega_h \big\} .
$$
Note that, by gluing in the standard way the degrees of freedom, the ensuing functions will have continuous values and continuous gradients across edges. Therefore the resulting space is indeed contained in $H^2(\Omega)$ and will yield a conforming solution. The \emph{global degrees of freedom} will simply be 
\begin{itemize}
\item Evaluation of $v_h$ at the vertexes of the mesh $\Omega_h$;
\item Evaluation of $\nabla v_h$ at the vertexes of the mesh $\Omega_h$.
\end{itemize}
Thus the dimension of $\Vh$ is three times the number of vertexes in the mesh. As a final note we observe that, in practice, it is recommended to scale the degrees of freedom $D2$ by some local characteristic mesh size $h_\nu$ in order to obtain a better condition number of the final system.

\subsection{Virtual forms}
\label{sec:vem2}

The second key step in the contruction of the method is the definition of suitable discrete forms.
Analogously to the finite element case, this forms will be constructed element by element and will depend on the degrees of freedom of the discrete space. Unlike in the finite element case, this forms will not be obtained by some Gauss integration of the shape functions (that are unknown inside the elements) but rather using the projection operators that we defined in the previous section.

We start by introducing a discrete approximation of the three exact local forms in \eqref{loc-forms-cont}. By making use of the projection operators of the previous section, the development of the bilinear forms follows a standard approach in the virtual element literature. We therefore refer, for instance, to \cite{volley} for more details and motivations regarding this construction.
Let $E\in\Omega_h$ be any element of the polygonal partition. We introduce the following (strictly) positive definite bilinear form on $\VE\times \VE$
$$
s_E(v_h,w_h) = \sum_{\stackrel{\nu \textrm{ vertexes}}{\textrm{of } \partial E}} 
\!\! \Big( v_h(\nu) \: w_h(\nu) + (h_\nu)^2 \: \nabla v_h(\nu) \cdot \nabla w_h(\nu) \Big)
\quad \forall v_h,w_h \in \VE ,
$$
where $h_\nu$ is some characteristic mesh size lenght associated to the node $\nu$ (for instance the maximum diameter among the elements having $\nu$ as a vertex). 

Recalling \eqref{loc-forms-cont}, we then propose the following discrete (and symmetric) local forms
\begin{equation}\label{loc-discr-forms}
\begin{aligned}
& a_{h,E}^\Delta(v_h, w_h) = a_{E}^\Delta(\Pi^\Delta_E v_h, \Pi^\Delta_E w_h) 
+ h_E^{-2} s_E(v_h - \Pi^\Delta_E v_h , w_h - \Pi^\Delta_E w_h) , \\
& a_{h,E}^\nabla(v_h, w_h) = a_{E}^\nabla(\Pi^\nabla_E v_h, \Pi^\nabla_E w_h) 
+ s_E(v_h - \Pi^\nabla_E v_h , w_h - \Pi^\nabla_E w_h) , \\
& a_{h,E}^0(v_h, w_h) = a_{E}^0(\Pi^0_E v_h, \Pi^0_E w_h) 
+ h_E^2 \, s_E(v_h - \Pi^0_E v_h , w_h - \Pi^0_E w_h) , 
\end{aligned}
\end{equation}
for all $v_h,w_h \in \VE$. 

The consistency of the discrete bilinear forms is assured by the first term on the right hand side of each relation, while the role of the second term $s_E(\cdot,\cdot)$ is only to guarantee the correct coercivity properties.
Indeed, noting that the projection operators appearing above are always orthogonal with respect to the associated bilinear form, it is immediate to check the following consistency lemma. 
\begin{lemma}[consistency]\label{lem:cons}
For all the three bilinear forms in \eqref{loc-discr-forms} it holds
$$
a_{h,E}^\dagger (p,v_h) = a_{E}^\dagger (p,v_h) \quad \forall p \in \P_2(E) , \ \forall v_h \in \VE ,
$$
where the symbol $\dagger$ stands for the symbol $\Delta,\nabla$ or $0$.
\end{lemma}

The lemma above states that the bilinear forms are exact whenever one of the two entries is a polynomial in $\P_2(E)$. 
In order to present a stability result for the proposed discrete bilinear forms, we need some mesh regularity assumptions on the mesh sequence $\{ \Omega_h \}_h$. 
\begin{assumption}\label{meshass}
We assume that there exist positive constants $c_s$ and $c_s'$ such that every element $E\in \{ \Omega_h \}_h$ is star shaped with respect to a ball with radius $\rho \ge c_s h_E$ and every edge $e \in \partial E$ has at least length $h_e \ge c_s' h_E$.
\end{assumption}

Under the above mesh regularity conditions, we can show the following lemma. Since the proof is standard and based on a scaling argument, it is omitted.
\begin{lemma}[stability]\label{lem:stab}
Let Assumption~\ref{meshass} hold. There exist two positive constants $c_\star, c^\star$ independent of the element $E\in \{ \Omega_h \}_h$ such that
$$
c_\star \, a_{E}^\dagger (v_h,v_h) \le 
a_{h,E}^\dagger (v_h,v_h) \le
c^\star a_{E}^\dagger (v_h,v_h) 
\quad \forall v_h \in \VE ,
$$
where the symbol $\dagger$ stands for the symbol $\Delta,\nabla$ or $0$.
\end{lemma}

Note that, as a consequence of the above lemma, it is immediate to check that the bilinear forms $a_{h,E}^\dagger(\cdot,\cdot)$ are continuous with respect to the relevant norm: $H^2$ for $\eqref{loc-discr-forms}_1$, $H^1$ for $\eqref{loc-discr-forms}_2$ and $L^2$ for $\eqref{loc-discr-forms}_3$.
The global discrete bilinear forms will be written (following the classical finite element procedure)
$$
a_h^\dagger(v_h,w_h) = \sum_{E\in\Omega_h} a_{h,E}^\dagger(v_h,w_h) 
\quad \forall v_h,w_h \in \Vh ,
$$
with the usual multiple meaning of the symbol $\dagger$.

\medskip

We now turn our attention to the semilinear form $r(\cdot;\cdot,\cdot)$, that we here write more explicitly:
$$
\begin{aligned}
& r(z;v,w) = \sum_{E\in\Omega_h} r_E(z;v,w) \quad \forall z,v,w \in H^2(\Omega) , \\
& r_E(z;v,w) = \int_E (3z(x)^2-1) \nabla v(x) \cdot \nabla w(x) \, {\rm d}x \quad \forall E\in\Omega_h .
\end{aligned}
$$
On each element $E$, we approximate the term $w(x)^2$ with its average, computed using the $L^2(E)$ bilinear form $a_{h,E}^0(\cdot,\cdot)$:
$$
(w_h^2)|_E \simeq |E|^{-1} a_{h,E}^0(w_h,w_h) ,
$$
where $|E|$ denotes the area of element $E$. This approach will turn out to have the correct approximation properties and, moreover, it preserves the positivity of $w^2$.
We therefore propose the following approximation of the local nonlinear forms
$$
r_{h,E}(z_h;v_h,w_h) =\widehat{\phi^\prime (z_h)}_{\vert E} \: a_{h,E}^\nabla(v_h,w_h)
\quad \forall z_h,v_h,w_h \in \VE 
$$
where 
$ \widehat{\phi^\prime (z_h)}_{\vert E} {=} 3 |E|^{-1} a_{h,E}^0(z_h,z_h) - 1$.
The global form is then assembled as usual
$$
r_h(z_h;v_h,w_h) = \sum_{E\in\Omega_h} r_{h,E}(z_h;v_h,w_h)
\quad \forall w_h,r_h,v_h \in \Vh .
$$

\subsection{Discrete problem}
\label{sec:vem3}
%
We here outline the Virtual Element discretization of problem \eqref{contpbl}, that follows a Galerkin approach in space combined with a backward Euler in time.   
Let us introduce the space with boundary conditions
$$
\Vhbc = \Vh \cap V =  \big\{ v \in \Vh \: : \: \partial_{\bf n} u = 0 \textrm{ on } \partial \Omega \big\} .
$$
As usual, it is convenient to first introduce the semi-discrete problem:
\begin{equation}\label{discrpbl}
\left\{
\begin{aligned}
& \textrm{ Find } u_h(\cdot, t) \textrm{ in } \Vhbc  \textrm{ such that} \\
& a_h^0(\partial_t u_h,v_h) + { \gamma^2} a_h^\Delta(u_h,v_h) + r_h(u_h,u_h;v_h) = 0 
\quad \forall v_h \in \Vhbc , \ a.e. \textrm{ in } (0,T) ,   \\
& u_h(0,\cdot)=u_{0,h}(\cdot) ,
\end{aligned}
\right.
\end{equation}
with $u_{0,h} \in \Vhbc$ a suitable approximation of $u_0$ and where the discrete forms above have been introduced in the previous section.\\

In order to introduce the fully discrete problem, we subdivide the time interval $[0,T]$ into $N$ uniform sub-intervals of length $k=T/N$ by selecting, as usual, the time nodes $0=t_0 < t_1 < ... < t_{N-1} < t_N=T$.
We now search for $\{ u_{hk}^1, u_{hk}^2,...,u_{hk}^N \}$ with $u_{hk}^i \in \Vhbc$ representing the solution at time $t_i$.

The fully discrete problem reads as follows:  Given $u_{hk}^0=u_{0,h} \in \Vhbc$, for $i=1, \ldots, N$ look for $u_{hk}^i \in \Vhbc$ such that
\begin{equation}\label{fullydisc}
\begin{aligned}
& k^{-1} a_h^0(u_{hk}^{i}-u_{hk}^{i-1},v_h) + { \gamma^2} a_h^\Delta(u_{hk}^{i},v_h) + r_h(u_{hk}^{i},u_{hk}^{i};v_h) = 0
&& \forall v_h \in \Vhbc.
\end{aligned}
\end{equation}

\section{Error analysis of the semi-discretization scheme}
\label{sec:error}
Throughout the subsequent discussion, we will employ the notation $x \lesssim y$ to denote the inequality $x \leq C y $ being $C$ a positive constant independent of the discretization parameters but that may depend on the regularity of the underlying continuous solution. Moreover, note that (unless needed to avoid confusion) in the sequel the dependence of $u$ and $u_h$ on time $t$ is left implicit and the bounds involving $u$ or $u_h$ hold for all $t \in (0,T]$.

In this section we present the convergence analysis of the semidiscrete Virtual Element formulation given in \eqref{discrpbl}. Our theoretical analysis will deal only with the semi-discrete case since the main novelty of the present paper is the (virtual element) space discretization. The 
error analysis of the fully discrete scheme follows from the analysis of the semi-discrete case employing standard techniques as for in the classical finite element case (see, e.g, \cite{Thomee-book}).

The subsequent convergence analysis will be performed under the following well accepted regularity assumption on the semi-discrete solution $u_h$ of  \eqref{discrpbl} (see, e.g., \cite{Elliott-French:1989} for a discussion on its validity).
\begin{assumption}\label{ass:uh}
The solution $u_h$ of  \eqref{discrpbl} satisfies $$ u_h \in L^\infty(\Omega) \quad \forall t \in (0,T].$$
\end{assumption}

As a starting point, we recall the following approximation result, see \cite{DupontScott} and \cite{VemSteklov,volley}.
\begin{proposition}\label{prop:interpolation}
Assume that Assumption~\ref{meshass}
is satised. Then for every $v \in H^{s}(E)$
there exists $v_{\pi} \in \P_k(E)$, $k\geq 0$ and $v_I \in \VE$ such that
\begin{equation}\label{eq:interpolation}
\begin{aligned}
&\seminorm{v-v_{\pi}}_{H^{\ell}(E)}
\lesssim h_K^{s-\ell} \seminorm{v}_{H^{s}(E)},
&& 1\leq s\leq k+1 , \ \ell=0,1,\ldots, s,   \\
& \seminorm{v-v_I}_{H^{\ell}(E)}
\lesssim h_K^{s-\ell} \seminorm{v}_{H^{s}(E)}
&& s=2,3, \ \ell=0,1,\ldots, s, 
\end{aligned}
\end{equation}
where the hidden constant depends only on
$k$ and on the constants in Assumption~\ref{meshass}.
\end{proposition}
Let $$ \overline{\phi^\prime(u)}_{\vert E} {=} 3 \vert E \vert^{-1} a_E^0(u,u)-1 $$
we define 
$$\overline{r}_h (u;v_h,w_h) {=} \sum_{E\in \Omega_h} \overline{\phi^\prime(u)}_{\vert E} a_{h,E}^\nabla (v_h,w_h).$$
We introduce the elliptic projection $P^h v \in W_h^0$ for $v\in H^4(\Omega)$ defined by
\begin{equation}\label{def:ellproj}
b_h(P^h v , \psi_h)= ({ \gamma^2}\Delta^2 v - \nabla \cdot (\phi^\prime (u) \nabla v) + \alpha v, \psi_h)
\end{equation}
for all $\psi_h \in W_h^0$, where $b_h(\cdot,\cdot)$ is the bilinear form 
\begin{equation}\label{def:bh}
b_h(v_h,w_h) = { \gamma^2 } a_h^{\Delta}(v_h,w_h) + \overline{r}_h (u;v_h,w_h) + \alpha (v_h,w_h)
\end{equation} 
being $\alpha$ a sufficiently large positive parameter.

For the subsequent analysis, it is instrumental to introduce the following auxiliary problem: find $\varphi \in V$ such that 
\begin{equation}\label{pb:aux}
b(\varphi,w)= (u-P^h u,w)_{H^1(\Omega)}
\end{equation}
for all $w\in V$, where $b(\cdot,\cdot)$ is the bilinear form 
\begin{equation}
b(v,w){=} { \gamma^2 } a^{\Delta}(v,w) + r(u;v,w) + \alpha (v,w) .
\end{equation}
We assume the validity of the following regularity result (see, e.g., \cite[Theorem A.1]{Elliott-French:1989} for a proof in the case a rectangular domain $\Omega$).
\begin{assumption}\label{ass:reg}
Let $\varphi$ be the solution of \eqref{pb:aux}. Then it holds 
\begin{equation}\label{adj:reg}
\| \varphi \|_{H^3(\Omega)} \leq C_\Omega \| u - P^h u \|_{H^1(\Omega)}
\end{equation}
with $P^h u$ be the elliptic projection defined in \eqref{def:ellproj} and where $C_\Omega$ is a positive constant only depending on $\Omega$.
\end{assumption}

We now collect some technical results that will be useful to prove the main result (Theorem \ref{conv_theo}).

\begin{lemma}\label{lm:semi-discr1}
Let $u$ be the solution to \eqref{contpbl} and $P^h u$ be the elliptic projection defined in \eqref{def:ellproj}. Then it holds 
\begin{eqnarray}
&&\| u - P^h u\|_{H^2(\Omega)} \lesssim h\label{proj:1}\\
&&\| u- P^h u\|_{H^1(\Omega)} \lesssim h^2.\label{proj:2}
\end{eqnarray} 
\end{lemma}
\begin{proof}
 It is worth observing that the solution $u$ to \eqref{contpbl} satifies 
\begin{equation}\label{aux:1}
b(u,\psi_h)= (\gamma^2 \Delta^2 u - \nabla \cdot (\phi^\prime (u) \nabla u) + \alpha u, \psi_h)
\end{equation}
for all $\psi_h\in W_h^0$. 

We first prove \eqref{proj:1}. Let $u_I\in W_h^0$ be a generic element to be made precise later. We preliminary remark that, using $P^h u - u_I \in W_h^0$ together with Lemma~\ref{lem:stab} and choosing $\alpha$ sufficiently large, we obtain 
\begin{eqnarray}
b_h(P^h u - u_I, P^h u - u_I) \gtrsim\| P^h u - u_I \|^2_{H^2(\Omega)}. \label{aux:0}
\end{eqnarray}
Moreover, employing \eqref{def:ellproj} and \eqref{aux:1} yields 
\begin{eqnarray}\label{aux:4}
b_h (P^h u , \psi_h) = (F,\psi_h) = b(u,\psi_h)\qquad \forall \psi_h \in W_h^0
\end{eqnarray}
with $F{=}{ \gamma^2} \Delta^2 u - \nabla \cdot (\phi^\prime (u) \nabla u) + \alpha u$.

Thus, using \eqref{aux:4} and letting $u_\pi$ be a discontinuous piecewise quadratic polynomial, we get 
\begin{eqnarray}
b_h(P^h u - u_I, P^h u -u_I) &=& b_h(P^h u, P^h u - u_I) - b_h(u_I, P^h u - u_I)\nonumber\\
&=&   b(u, P^h u - u_I) -  b_h(u_\pi, P^h u - u_I) +  b_h(u_\pi - u_I , P^h u - u_I) \nonumber\\
&=&  b(u, P^h u - u_I) -  \overline{b}(u_\pi, P^h u - u_I) +  b_h(u_\pi - u_I , P^h u - u_I)\nonumber
\end{eqnarray}  
where in the last equality we apply the consistency result contained in Lemma~\ref{lem:cons} to the bilinear form $$\overline{b}(v,w){=} \sum_{E\in \Omega_h} { \gamma^2 }a_E^{\Delta}(v,w) + \overline{\phi^\prime (u)}_{\vert E} \ a_E^\nabla (v,w) + \alpha\,  a_E^0(v,w).$$

From the above identity, using \eqref{aux:0} we get 
\begin{eqnarray}
\| P^h u - u_I \|^2_{H^2(\Omega)} &\lesssim&  b(u, P^h u - u_I) -  \overline{b}(u, P^h u - u_I) 
 +  \overline{b}(u - u_\pi, P^h u - u_I)\nonumber\\
&& +\, b_h(u_\pi - u_I , P^h u - u_I).\label{aux:2}
\end{eqnarray} 
Let us now estimate each term on the right hand side of \eqref{aux:2}. From the definitions of the bilinear forms $b(\cdot,\cdot) $ and $\overline{b}(\cdot,\cdot)$, and {employing the interpolation estimates given in Proposition~\eqref{prop:interpolation}}, we obtain
\begin{eqnarray}
 b(u, P^h u - u_I) -  \overline{b}(u, P^h u - u_I)  &=& \sum_{E\in \Omega_h} \int_E (\phi^\prime (u) - \overline{\phi^\prime (u)}_{\vert E}) \nabla u \cdot  \nabla (P^h u - u_I) dx \nonumber\\
 &\lesssim& h \| P^h u - u_I\|_{H^2(\Omega)} \label{aux:5}.
\end{eqnarray} 
Moreover, choosing $u_I$ and $u_\pi$ such that (see Proposition~\ref{prop:interpolation})
\begin{equation}\label{aux:7}
\| u- u_I\|_{H^2(E)} + \| u- u_\pi\|_{H^2(E)} \lesssim h
\end{equation} 
and employing the continuity properties of $\overline{b}(\cdot,\cdot)$ and $b_h(\cdot,\cdot)$ we get 
\begin{equation}\label{aux:6}
\overline{b}(u - u_\pi, P^h u - u_I) + b_h(u_\pi - u_I , P^h u - u_I) \lesssim h \| P^h u - u_I\|_{H^2(\Omega)} .
\end{equation}
Substituting \eqref{aux:5} and \eqref{aux:6}  in \eqref{aux:2} and using triangle inequality together with \eqref{aux:7} we get \eqref{proj:1}.

We now prove \eqref{proj:2}. Taking $w=u-P^h u$ in \eqref{pb:aux} yields 
\begin{equation}\label{eq:basic}
\| u - P^h u\|^2_{H^1(\Omega)} = b(\varphi,u-P^h u)=b(\varphi - \varphi_I, u-P^h u) + b(\varphi_I, u- P^h u)
\end{equation}
We now estimate each term on the right hand side of the above equation.
Choosing, accordingly to Proposition~\eqref{prop:interpolation}, $\varphi_I$ such that $\| \varphi -\varphi_I\|_{H^2(\Omega)} \lesssim h$, using \eqref{adj:reg} and employing the continuity property of the bilinear form $b(\cdot,\cdot)$ together with \eqref{proj:1} we get 
\begin{equation}\label{eq:A0}
b(\varphi-\varphi_I,u-P^h u)\lesssim \| \varphi -\varphi_I\|_{H^2(\Omega)} \| u  -P^h u\|_{H^2(\Omega)} \lesssim h^2 
\| u - P^h u\|_{H^1(\Omega)}.  
\end{equation}
Using \eqref{aux:1} with $\varphi_I\in W_h^0$ we get $b(\varphi_I,u) = b_h(\varphi_I , P^h u)$ which implies 
\begin{eqnarray}
b(\varphi_I,u-P^h u)&=& b_h(\varphi_I, P^h u) - b(\varphi_I,P^h u)\nonumber\\
&=& { \gamma^2 } (a^\Delta_h(\varphi_I,P^h u) - a^\Delta (\varphi_I, P^h u)) + \overline{r}_h(u;\varphi_I,P^h u) - r(u;\varphi_I,P^h u)\nonumber\\
&{=}&{ \gamma^2 }A_1+A_2.\nonumber  
\end{eqnarray}
Let $u_\pi$ and $\varphi_\pi$ be piecewise discontinuous quadratic polynomials such that $\|u-u_\pi\|_{H^2(E)} \lesssim h$ and $\|\varphi-\varphi_\pi\|_{H^2(E)} \lesssim h$. Applying twice the consistency result contained in Lemma~\ref{lem:cons}  together with \eqref{proj:1} we obtain 
\begin{eqnarray}
A_1 &=&  \sum_{E\in \Omega_h} a_{h,E}^\Delta(\varphi_I - \varphi_\pi, P^h u- u_\pi) - \sum_{E\in \Omega_h} a^\Delta_E (\varphi_I - \varphi_\pi, P^h u- u_\pi)\nonumber\\
&\lesssim& (\|\varphi-\varphi_I\|_{H^2(E)} + \|\varphi-\varphi_\pi\|_{H^2(E)})( \|P^h u - u \|_{H^2(E)} + \|u -u_\pi\|_{H^2(E)})\nonumber\\
&\lesssim& h^2 \vert \varphi \vert_{H^3(\Omega)}   \vert u \vert_{H^3(\Omega)}\lesssim h^2 \| u- P^h u\|_{H^1(\Omega)}.\label{eq:A1}
\end{eqnarray}
Let us now estimate the term $A_2$. Using the definitions of $r(\cdot;\cdot,\cdot)$ and $\overline{r}_h(\cdot;\cdot,\cdot)$ we get 
\begin{eqnarray}
A_2&=&\sum_{E\in \Omega_h} \overline{\phi^\prime (u)}_{\vert E} \left ( a_{h,E}^\nabla (\varphi_I,P^h u) -
a_{E}^\nabla (\varphi_I,P^h u)\right  ) + \int_E \left (\overline{\phi^\prime(u)}_{\vert E}-\phi^\prime (u)\right ) \nabla \varphi_I \cdot \nabla P^h u  \nonumber \\
&=:& A_{2,1} + A_{2,2}.\nonumber
\end{eqnarray}
Proceeding as in the bound of $A_1$ and employing assumption \eqref{CH-regularity} on the regularity if $u$ we obtain 
\begin{equation}\label{eq:A21}
A_{2,1} \lesssim h^2 \| u- P^h u\|_{H^1(\Omega)}.
\end{equation}
Finally, we estimate the term $A_{2,2}$. By employing the orthogonality property of projectors and 
denoting by $\overline{(\cdot)}$ the projection of $(\cdot)$ on constants we get 
\begin{eqnarray}
A_{2,2}&=& \sum_{E\in \Omega_h} \int_E \left (\overline{\phi^\prime(u)}_{\vert E}-\phi^\prime (u)\right ) \left (\nabla \varphi_I \cdot \nabla P^h u - \overline{\nabla \varphi} \cdot \overline{\nabla u}\right)\, dx\nonumber\\
&=&\sum_{E\in \Omega_h} 
\int_E  \left( \overline{\phi^\prime(u)}_{\vert E}-\phi^\prime (u)\right) 
( \nabla \varphi_I -  \overline{\nabla \varphi} )\cdot \nabla P^h u \, dx\nonumber\\
&&+ \int_E \left (\overline{\phi^\prime(u)}_{\vert E}-\phi^\prime (u)\right ) \overline{\nabla \varphi} 
\cdot (\nabla P^h u - \overline{\nabla u})\, dx.
\end{eqnarray}
{Using the interpolation estimates given in Proposition~\eqref{prop:interpolation}}, and  employing  \eqref{adj:reg} and \eqref{proj:1} together with 
the following inequalities
\begin{eqnarray}
&&\|\nabla \varphi_I - \overline{\nabla \varphi } \|_{L^2(E)} \leq 
\|\nabla \varphi_I - \nabla \varphi \|_{L^2(E)} + \|\nabla \varphi- \overline{\nabla \varphi} \|_{L^2(E)} \lesssim h 
\| \varphi \|_{H^2(E)} \nonumber	\\
&&\| \nabla P^h u \|_{L^2(E)}\leq \| \nabla P^h u - \nabla u \|_{L^2(E)} + \| \nabla  u \|_{L^2(E)}\lesssim (1+h ) 
\| u \|_{H^2(E)}\nonumber\\
&&\| \overline{\nabla \varphi} \|_{L^2(E)}\leq \| \overline{\nabla \varphi_I} - {\nabla \varphi}\|_{L^2(E)} + 
\| {\nabla \varphi} \|_{L^2(E)} \lesssim (1+h ) \| \varphi \|_{H^2(E)}\nonumber\\
&&\| \nabla P^h u - \overline{\nabla u}\|_{L^2(E)} = \| \nabla (P^h u - u) + (\nabla u - \overline{\nabla u})\|_{L^2(E)}\lesssim  \| P^h u - u\|_{H^2(E)}  + h \| u \|_{H^2(E)} \nonumber\\
&& \| \overline{\phi^\prime(u)}_{\vert E}-\phi^\prime (u) \|_{L^\infty(E)} \lesssim h |\phi^\prime (u)|_{W^{1,\infty}(E)} , \nonumber
\end{eqnarray}
we obtain 
\begin{equation}\label{eq:A22}
A_{2,2} \lesssim h^2 \| P^h u -u \|_{H^1(\Omega)}. 
\end{equation}
Combining \eqref{eq:A1}, \eqref{eq:A21}, \eqref{eq:A22}, \eqref{eq:A0} with \eqref{eq:basic} we obtain \eqref{proj:2}.
\end{proof}

\begin{lemma}\label{lm:semi-discr2}
Let $u$ be the solution to \eqref{contpbl} and $P^h u$ be the elliptic projection defined in \eqref{def:ellproj}. Then it holds 
\begin{eqnarray}
&&\| u_t - (P^h u)_t\|_{H^2(\Omega)} \lesssim h\label{proj:2.1}\\
&&\| u_t- (P^h u)_t\|_{H^1(\Omega)} \lesssim h^2.\label{proj:2.2}
\end{eqnarray} 
\end{lemma}
\begin{proof}
It is sufficient to observe that it holds
$$b_h ((P^h u)_t , \psi_h) = b(u_t,\psi_h) + (\phi^{\prime\prime}(u) u_t \nabla u,\nabla \psi_h) - 
\sum_{E\in \Omega_h} \partial_t \overline{(\phi^\prime (u))}_{\vert E} a_{h,E}^\nabla (P^h u, \psi_h)$$
for all $\psi_h \in W_h^0$. Then proceeding as in Lemma~\ref{lm:semi-discr1} and using 
$$\| \partial_t \overline{(\phi^\prime (u))}_{\vert E} -\phi^{\prime\prime} (u)u_t \|_{L^\infty(E)}= 
\| \overline{6 u u_t} - 6 u u_t\|_{L^\infty(E)}\lesssim h$$
we obtain the thesis.
\end{proof}

\begin{lemma}\label{lm:pert}
Let $u$ be the solution to \eqref{contpbl} and $P^h u$ be the elliptic projection defined in \eqref{def:ellproj}. 
Then, setting $\rho=u-P^h u$ and $\theta=P^h u - u_h$,  it holds 
\begin{equation}
 r_h(u_h,u_h;\theta)- \overline{r}_h (u;P^h u,\theta) \lesssim \vert \theta\vert_{H^1(\Omega)} \left(\|\theta\|_{L^2(\Omega)} + \| \rho\|_{L^2(\Omega)} + \vert \theta \vert_{H^1(\Omega)} + h^2 \right) .
\end{equation}
\end{lemma}
\begin{proof}
We preliminary observe that using Lemma~\ref{lm:semi-discr1} and~\ref{lm:semi-discr2} and proceeding as in the proof of \cite[(3.2c)]{Elliott-French:1989} yield $P^h u \in W^{1,\infty}(\Omega)$, with norm bounded uniformly in time.  Moreover, it holds
\begin{eqnarray}
r_h(u_h;u_h,\theta)- \overline{r}_h (u;P^h u,\theta)&=& r_h(u_h;P^h u,\theta)- \overline{r}_h (u;P^h u,\theta) + 
 r_h(u_h;u_h-P^h u,\theta)\nonumber\\
 &=&\sum_{E\in \Omega_h} (\widehat{\phi^\prime(u_h)} - \overline{\phi^\prime (u)})_{\vert E} a_{h,E}^\nabla (P^h u, \theta) \nonumber\\
 &&\quad +\,  r_h(u_h;u_h-P^h u,\theta)\nonumber\\
 &=&A+B. \nonumber
\end{eqnarray}
Let us first estimate the term $A$ which can be written as follows
\begin{eqnarray}
A&=& \sum_{E\in \Omega_h}  ( \widehat{\phi^\prime(u_h)} - \widehat{\phi^\prime(P^h u)} + 
\widehat{\phi^\prime(P^h u)} -  \overline{\phi^\prime (P^h u)} +   
\overline{\phi^\prime (P^h u)}  - \overline{\phi^\prime (u)})_{\vert E}  a_{h,E}^\nabla (P^h u, \theta)
\nonumber\\
&=&\sum_{E\in \Omega_h}  ( I + II + III)_{\vert E}  a_{h,E}^\nabla (P^h u, \theta).\nonumber
\end{eqnarray}
Using Lemma~\ref{lem:cons} we obtain  
\begin{eqnarray}
A &\lesssim& \sum_{E\in \Omega_h}  (|I| + |II| + |III|)_{\vert E} \| P^h u \|_{H^1(E)}\vert \theta\vert_{H^1(E)}\nonumber\\
&\lesssim& \| P^h u \|_{W^{1,\infty}(\Omega)}\sum_{E\in \Omega_h}  \vert E\vert^{1/2} (|I| + |II| + |III|)_{\vert E} 
\vert \theta\vert_{H^1(E)}\nonumber\\
&\lesssim& \| P^h u \|_{W^{1,\infty}(\Omega)}
\left(\sum_{E\in \Omega_h}  \vert E\vert (I^2 + II^2 + III^2)_{\vert E}\right)^{1/2} \vert \theta\vert_{H^1(\Omega)} 
\nonumber\\
&\lesssim& 
 \| P^h u \|_{W^{1,\infty}(\Omega)} (A_I+A_{II}+A_{III})  |\theta|_{H^1(\Omega)} .\label{aux:11}
\end{eqnarray}
where $A_{(\cdot)}{=} \left(\sum_{E\in \Omega_h}  \vert E\vert (\cdot)^2_{\vert E}\right)^{1/2}$.
Using the definition of $\widehat{(\cdot)}$ and Lemma~\ref{lem:cons} we obtain
\begin{eqnarray}
I&=&\frac{3}{\vert E\vert} (a_{h,E}^0 (u_h,u_h)-a_{h,E}^0(P^h u,P^h u))\nonumber\\
&=& \frac{3}{\vert E\vert} (a_{h,E}^0 (u_h-P^h u,u_h+P^h u)) \nonumber\\
&\lesssim & \frac{3}{\vert E\vert} \| u_h-P^h u\|_{L^2(E)} (\| u_h\|_{L^2(E)}+\|P^h u\|_{L^2(E)})
\end{eqnarray}
which implies
\begin{eqnarray}
A_I &\lesssim& \sum_{E\in \Omega_h} \left(\frac{1}{\vert E\vert} \|\theta\|^2_{L^2(E)} 
\left (\| u_h\|_{L^2(E)}+\|P^h u\|_{L^2(E)}\right)^2\right)^{1/2}\nonumber\\
&\lesssim& (\|u_h\|_{L^\infty(\Omega)} + \|P^h u\|_{L^\infty(\Omega)})\left( \sum_{E\in \Omega_h} 
\| \theta\|^2_{L^2(E)}\right)^{1/2}\nonumber\\
&\lesssim& \| \theta\|_{L^2(\Omega)}
\end{eqnarray}
where in the last step we employed Assumption \ref{ass:uh} on the regularity of $u_h$.

Similarly, using the definition of $\overline{(\cdot)}$ we have 
\begin{eqnarray}
III=\frac{3}{\vert E\vert} \left( \int_E (P^h u)^2 -\int_E u^2 \right)\leq \frac{3}{\vert E\vert} 
\| P^h u -u\|_{L^2(E)}\left( \| u\|_{L^2(E)} + \| P^h u\|_{L^2(E)} \right)\nonumber
\end{eqnarray}
which yields
\begin{eqnarray}
A_{III}\lesssim (\|u_h\|_{L^\infty(\Omega)} + \|P^h u\|_{L^\infty(\Omega)})\left( \sum_{E\in \Omega_h} 
\| P^h u - u \|^2_{L^2(E)}\right)^{1/2}\lesssim\| P^h u - u \|_{L^2(\Omega)}.\nonumber
\end{eqnarray}
Finally, employing $q \in \mathbb{P}_2(E)$  together with Lemma~\ref{lem:cons} and {the interpolation estimates of Proposition~\ref{prop:interpolation}}, it is easy to prove that the following holds
\begin{eqnarray}
II&=&\frac{3}{\vert E \vert} \left( a_{h,E}^0 (P^h u,P^h u) - (P^h u,P^h u) \right) = 
\frac{3}{\vert E \vert} \left( a_{h,E}^0 (P^h u - q,P^h u) - (P^h u - q,P^h u) \right) \nonumber\\
&\lesssim& \frac{1}{\vert E \vert} \|P^h u - q\|_{L^2(E)} \| P^h u\|_{L^2(E)} 
\lesssim  \frac{1}{\vert E \vert} \left ( \|P^h u - u\|_{L^2(E)} + \|u - q\|_{L^2(E)} \right)\| P^h u\|_{L^2(E)}\nonumber\\
&\lesssim&  \frac{1}{\vert E \vert} \left ( \|P^h u - u\|_{L^2(E)} + h^2 \right)\| P^h u\|_{L^2(E)}\nonumber
\end{eqnarray}
which implies 
\begin{eqnarray}
A_{II} &\lesssim& \| P^h u\|_{L^\infty (\Omega)} \left (\left ( \sum_{E\in \Omega_h}  \|P^h u - u\|^2_{L^2(E)} \right)^{1/2}  + h^2\right)\nonumber\\
&\lesssim&  \|P^h u - u\|_{L^2(\Omega)} + h^2. 
\end{eqnarray}
Employing the above estimates for $A_I,A_{II}$ and $A_{III}$ into \eqref{aux:11} and recalling that $\| P^h u \|_{W^{1,\infty}(\Omega)}$ is uniformly bounded in time, we get
\begin{eqnarray}
A \lesssim  \left(\|\theta\|_{L^2(\Omega)} + \| P^h u -u \|_{L^2(\Omega)} + h^2 \right)  \vert \theta\vert_{H^1(\Omega)}.
\end{eqnarray}
To conclude it is sufficient to estimate $B$. Using the definition of $r_h(\cdot;\cdot,\cdot)$ together with  Lemma~\ref{lem:stab} and Assumption \ref{ass:uh} we have 
\begin{eqnarray}
B\lesssim \|u_h\|_{L^\infty(\Omega)} \vert \theta \vert^2_{H^1(\Omega)}\lesssim  \vert \theta \vert^2_{H^1(\Omega)}.
\end{eqnarray}
\end{proof}
\begin{lemma}\label{lm:norm}
Let $v_h\in W_h^0$ and $\varepsilon >0$. Then  there exists a constant $C_\varepsilon$ depending on $\epsilon$ such that it holds 
\begin{equation}
\vert v_h \vert^2_{H^1(\Omega)} \leq \varepsilon \vert v_h \vert^2_{H^2(\Omega)} + C_\varepsilon \| v_h \|^2_{L^2(\Omega)}. 
\end{equation}
\end{lemma}
\begin{proof}
It is straightforward to observe that it holds
\begin{eqnarray}
\vert v_h \vert^2_{H^1(\Omega)} &=& \sum_{E} \int_E \nabla v_h \cdot \nabla v_h =
\sum_{E} \left\{ \int_{\partial E} v_h \frac{\partial v_h}{\partial n} \, ds - \int_E \Delta v_h   v_h\, dx \right\} \nonumber\\
&=&\sum_{E} \left\{ - \int_E \Delta v_h   v_h\, dx \right\} \leq \varepsilon \| \Delta v_h \|_{L^2(\Omega)}^2 + C_\varepsilon \| v_h \|^2_{L^2(\Omega)} 
\end{eqnarray} 
where we used Cauchy-Schwarz inequality and the fact that $W_h^0 \subset H^2(\Omega)$.
\end{proof}

We are now ready to prove the following convergence result. 
\begin{theorem}\label{conv_theo}
Let $u$ be the solution to \eqref{contpbl} and $u_h$ the solution to  \eqref{discrpbl}. Then for all $t \in [0,T]$ it holds 
\begin{equation}
\| u - u_h \|_{L^2(\Omega)} \lesssim h^2.
\end{equation}
\end{theorem}

\begin{proof}
As usual, the argument is based on the following error decomposition 
\begin{equation}
u-u_h= (u- P^h u) + (P^h u - u_h)=: \rho +\theta.
\end{equation}
In view of Lemma~\ref{lm:semi-discr1}, we only need to estimate $\| \theta\|_{L^2(\Omega)}$.
Proceeding as in \cite{Elliott-French:1989}, we first observe that it holds 
\begin{eqnarray}
a_h^0(\theta_t,\chi_h)+ { \gamma^2 } a_h^\Delta(\theta,\chi_h)&=&a_h^0((P^h u - u_h)_t, \chi_h) + { \gamma^2 } a^\Delta_h(P^h u -u_h,\chi_h)\nonumber\\
&=&a_h^0((P^h u)_t,\chi_h) + { \gamma^2 } a_h^\Delta (P^h u, \chi_h) \nonumber\\&&\ - [a_h^0((u_h)_t,\chi_h)+ { \gamma^2 } a^\Delta_h(u_h,\chi_h)]\nonumber\\
&=& a_h^0((P^h u)_t,\chi_h) + { \gamma^2 } a_h^\Delta (P^h u, \chi_h) +  r_h(u_h,u_h;\chi_h).\nonumber
\end{eqnarray}
Using  \eqref{def:bh} and \eqref{def:ellproj} it holds 
\begin{eqnarray}
{ \gamma^2 } a_h^{\Delta}(P^h u,\chi_h) &=& b_h(P^h u,\chi_h) - \overline{r}_h (u;P^h u,\chi_h) - \alpha (P^h u,\chi_h)\nonumber\\
&=& ({ \gamma^2 }\Delta^2 u - \nabla \cdot (\phi^\prime (u) \nabla u) + \alpha u, \chi_h) - \overline{r}_h (u;P^h u,\chi_h) - \alpha (P^h u,\chi_h)\nonumber\\
&=& ({ \gamma^2 }\Delta^2 u - \nabla \cdot (\phi^\prime (u) \nabla u) , \chi_h) - \overline{r}_h (u;P^h u,\chi_h) + \alpha (\rho,\chi_h).\nonumber
\end{eqnarray}
Thus, we have 
\begin{eqnarray}
a_h^0(\theta_t,\chi_h)+ { \gamma^2 } a_h^\Delta(\theta,\chi_h)&=& a_h^0((P^h u )_t, \chi_h)  + ({ \gamma^2 } \Delta^2 u - \nabla \cdot (\phi^\prime (u) \nabla u) , \chi_h)\nonumber\\
&&  + \,  r_h(u_h,u_h;\chi_h)- \overline{r}_h (u;P^h u,\chi_h) + \alpha (\rho,\chi_h)\nonumber\\
&=& - a_h^0(\rho_t, \chi_h) + (u_t + { \gamma^2 } \Delta^2 u - \nabla \cdot (\phi^\prime (u) \nabla u) , \chi_h)\nonumber\\ 
&&  + \,  r_h(u_h,u_h;\chi_h)- \overline{r}_h (u;P^h u,\chi_h) + \alpha (\rho,\chi_h)\nonumber\\
&=&  \alpha (\rho,\chi_h) - a_h^0(\rho_t, \chi_h) + \,  r_h(u_h,u_h;\chi_h)- \overline{r}_h (u;P^h u,\chi_h).\nonumber
\end{eqnarray}
Taking $\chi_h=\theta$ in the above equality we get 
\begin{eqnarray}
a_h^0(\theta_t,\theta)+ { \gamma^2 } a_h^\Delta(\theta,\theta)=  \alpha (\rho,\theta) - a_h^0(\rho_t, \theta) + \,  r_h(u_h,u_h;\theta)- \overline{r}_h (u;P^h u,\theta)
\end{eqnarray}
which, combined with the stability properties of $a_h^\Delta (\cdot,\cdot)$ and $a_h^0 (\cdot,\cdot)$ (see Lemma~\ref{lem:stab}), implies the following crucial inequality
\begin{eqnarray}
\frac{1}{2}\frac{d}{dt} \|\theta\|^2_{L^2(\Omega)} +{ \gamma^2 } \vert \theta \vert^2_{H^2(\Omega)} \lesssim   (\alpha \|\rho\|_{L^2(\Omega)} + \| \rho_t \|_{L^2(\Omega)} )\| \theta\|_{L^2(\Omega)} +  r_h(u_h,u_h;\theta)- \overline{r}_h (u;P^h u,\theta).\nonumber
\end{eqnarray}
Employing Lemmas~\ref{lm:semi-discr1},~\ref{lm:semi-discr2},~\ref{lm:pert} and~\ref{lm:norm} we obtain 
\begin{equation}
\frac{1}{2}\frac{d}{dt} \|\theta\|^2_{L^2(\Omega)} + { \gamma^2 } \vert \theta \vert^2_{H^2(\Omega)} \lesssim  h^4 + \| \theta \|^2_{L^2(\Omega)}
\end{equation}
which, combined with Gronwall's lemma, yields the required estimate for $\| \theta \|_{L^2(\Omega)}$.

\end{proof}

\section{Numerical results}
\label{sec:num}

The time discretization is performed by the Backward Euler method. The resulting
non-linear system (\ref{fullydisc}) at each time step is solved by the Newton method, using
the $l^2$ norm of the relative residual as a stopping criterion. The tolerance for
convergence is $1e-6$. For the simulations, we have used a Matlab code and a Fortran90
parallel code based on the PETSc library. The parallel tests were run on the 
FERMI linux cluster of the CINECA consortium (www.cineca.it).

\subsection{Test 1: convergence to exact solution}
\begin{table}[htbp]
\caption{Test 1: \noindent $H^2$, $H^1$ and $L^2$ errors and convergence rates $\alpha$ computed on four quadrilateral 
meshes discretizing the unit square.}
\begin{center}
\begin{tabular}{c|cc|cc|cc}
\hline
$h$    &$|u_h-u|_{H^2(\Omega)}$  &$\alpha$  &$|u_h-u|_{H^1(\Omega)}$  &$\alpha$  &$||u_h-u||_{L^2(\Omega)}$  &$\alpha$\\
\hline
1/16   &1.35e-1          &--        &8.57e-2          &--        &8.65e-2            &--\\
1/32   &5.86e-2          &1.20      &2.20e-2          &1.96      &2.20e-2            &1.97\\
1/64   &2.79e-2          &1.07      &5.53e-3          &1.99      &5.52e-3            &1.99\\
1/128  &1.38e-2          &1.02      &1.37e-3          &2.01      &1.37e-3            &2.01\\
\hline
\end{tabular}
\label{tab_test1}
\end{center}
\end{table}

In this test, we study the convergence of our VEM discretization applied to the Cahn-Hilliard equation
with a forcing term $f$ obtained imposing as exact solution
$u(x,y,t)=t\cos(2\pi x)\cos(2\pi y)$. The parameter $\gamma$ is set to $1/10$ and 
the time step size $\Delta t$ is $1e-7$.
The $H^2$, $H^1$ and $L^2$ errors are computed at $t=0.1$ on four quadrilateral 
meshes discretizing the unit square. 

The results reported in Table~\ref{tab_test1} show that in the $L^2$ norm the VEM method 
converges with order 2, as predicted by Theorem~\ref{conv_theo}. In the $H^2$ and $H^1$ norms, 
the method converges with order 1 and 2 respectively, as can be expected according to the FEM theory and the approximation properties of the adopted virtual space. 

\subsection{Test 2: evolution of an ellipse}
In this test, we consider the Cahn-Hilliard
equation on the unit square with $ \gamma = 1/100$. The time step size $\Delta t$ is $5e-5$. 
The initial datum
$u_0$  is a piecewise constant function whose jump-set is an ellipse:
\[
u_0(x,y)=
\left\{
\begin{array}{ll}
0.95 & \mbox{ if } 9(x-0.5)^2 + (y-0.5)^2 < 1/9,\\
-0.95 & \mbox{ otherwise }.
\end{array}
\right.
\]
Both a structured quadrilateral mesh and an unstructured triangular mesh (generated with the 
mesh generator of the Matlab PDEToolbox) are considered, with 49923 and
13167 dof, respectively.
As expected the initial datum $u_0$ with the ellipse-shaped jump-set evolves to a
steady state exhibiting a circular interface; see Figs~\ref{ellip_quad} and
\ref{ellip_tri}. Thereafter, no motion will
occur as the interface has constant curvature.
\begin{figure}[t]
\begin{center}
\subfigure[Quadrilateral mesh of $16384=128\times 128$ elements  ($49923$ degrees of freedom). \label{ellip_quad}]{
\includegraphics[width=0.32\textwidth]{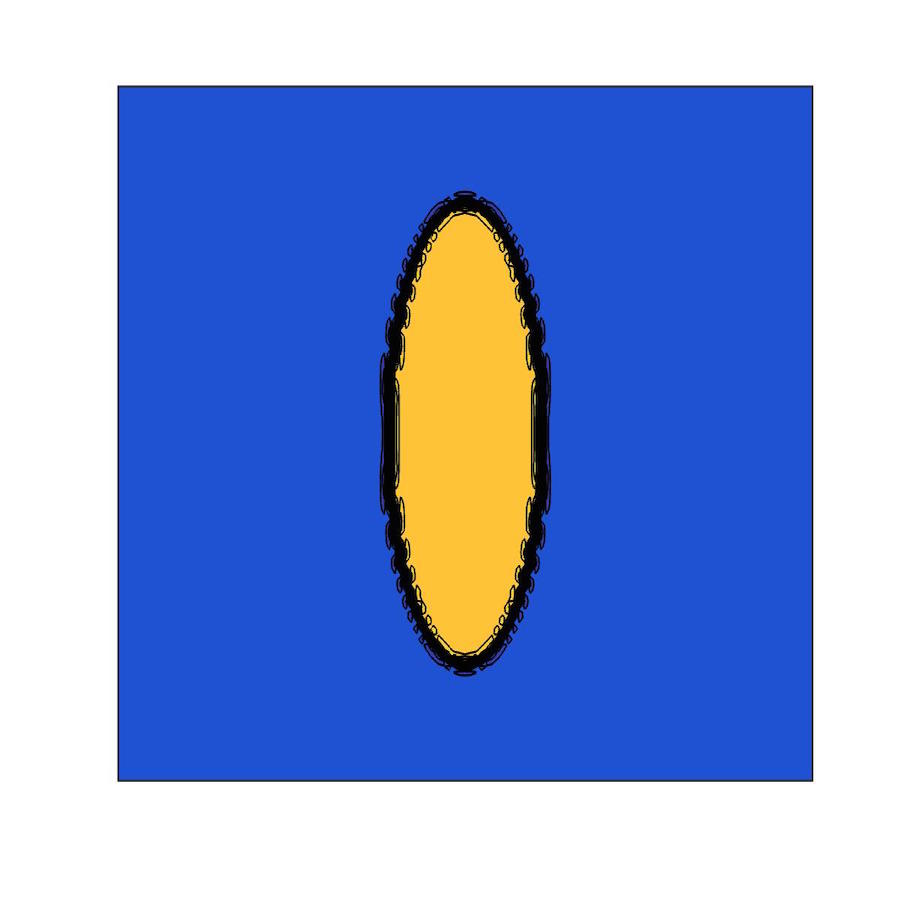}
\includegraphics[width=0.32\textwidth]{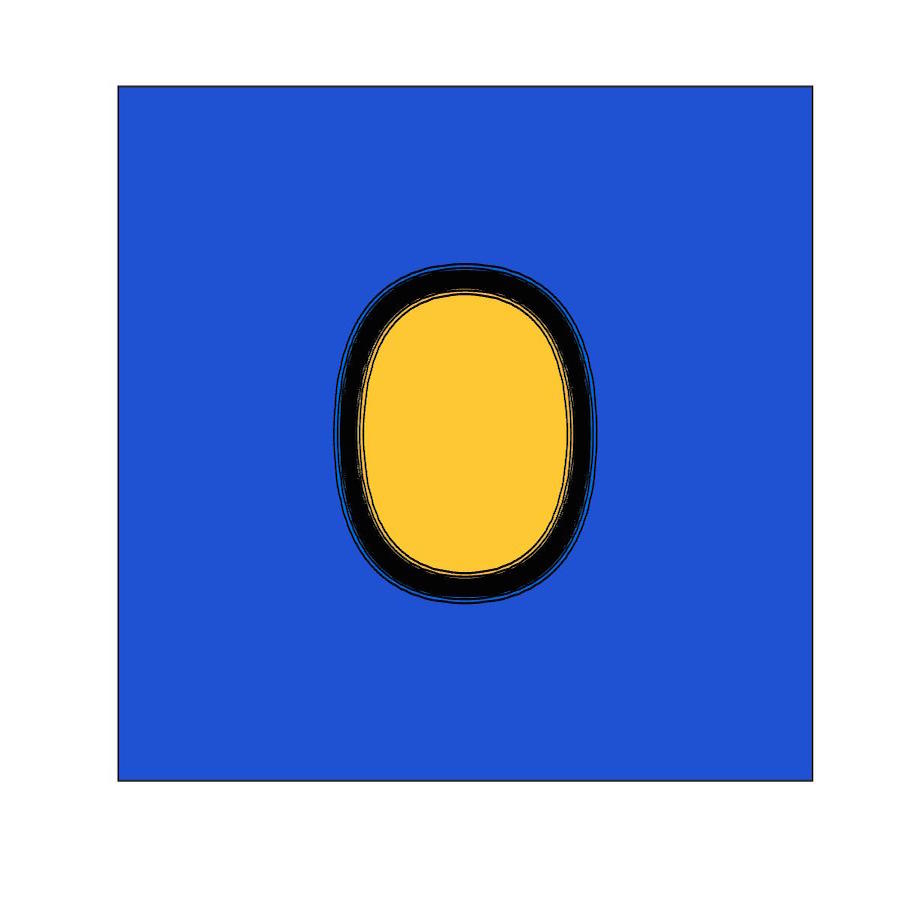}
\includegraphics[width=0.32\textwidth]{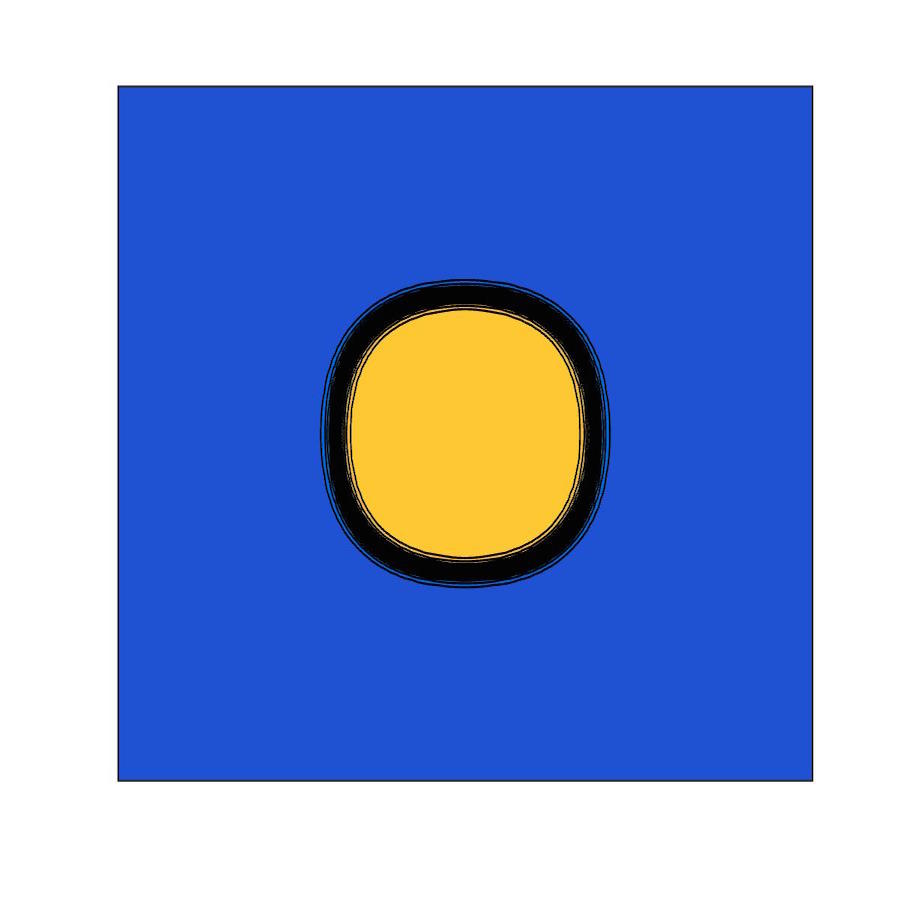}}
\subfigure[Triangular mesh of $8576$ elements
($13167$ degrees of freedom). \label{ellip_tri}]{
\includegraphics[width=0.32\textwidth]{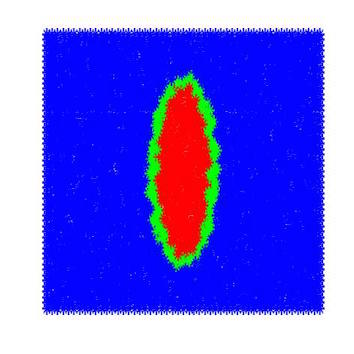}
\includegraphics[width=0.32\textwidth]{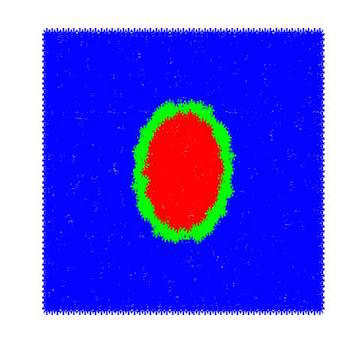}
\includegraphics[width=0.32\textwidth]{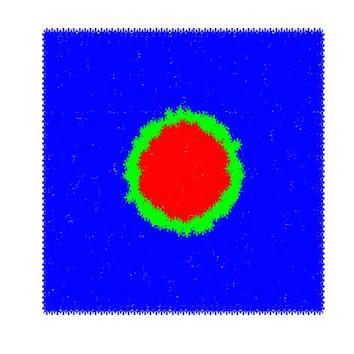}}
\caption{Test 2: evolution of an ellipse at three temporal frames  ($t=0, 0.5, 1$).}
\end{center}
\end{figure}
%
%
%
%

\subsection{Test 3: evolution of a cross}
We use here the same domain and the parameters as in Test 2. The initial datum $u_0$ is
a piecewise constant function whose jump-set has the shape of a cross; 
see Figs.~\ref{cross_quad},~\ref{cross_tri} and~\ref{cross_paul} (t = 0).
The same quadrilateral and triangular meshes of the Test 2 are considered,
with 49923 and 13167 dof, respectively, and a Voronoi polygonal mesh 
(including quadrilaterals, pentagons and hexagons, see Fig. \ref{paul_example} as example) with 59490 dof.
As in the ellipse example, the initial datum $u_0$ with a cross-shaped
jump-set evolves to a steady state exhibiting a circular interface, see
Figs.~\ref{cross_quad},~\ref{cross_tri} and~\ref{cross_paul}.
\begin{figure}[t]
\begin{center}
\includegraphics[width=4cm]{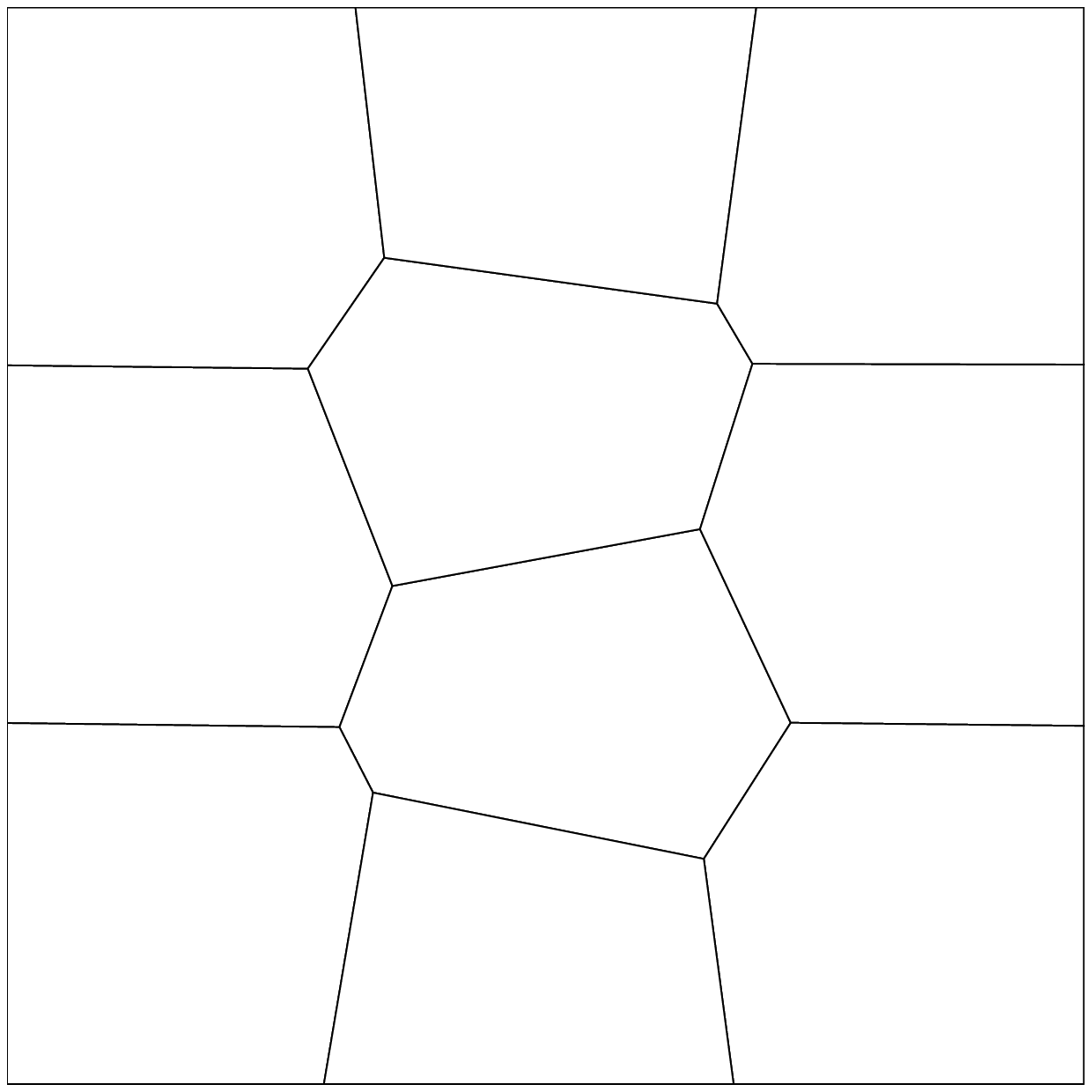}
\includegraphics[width=4.4cm]{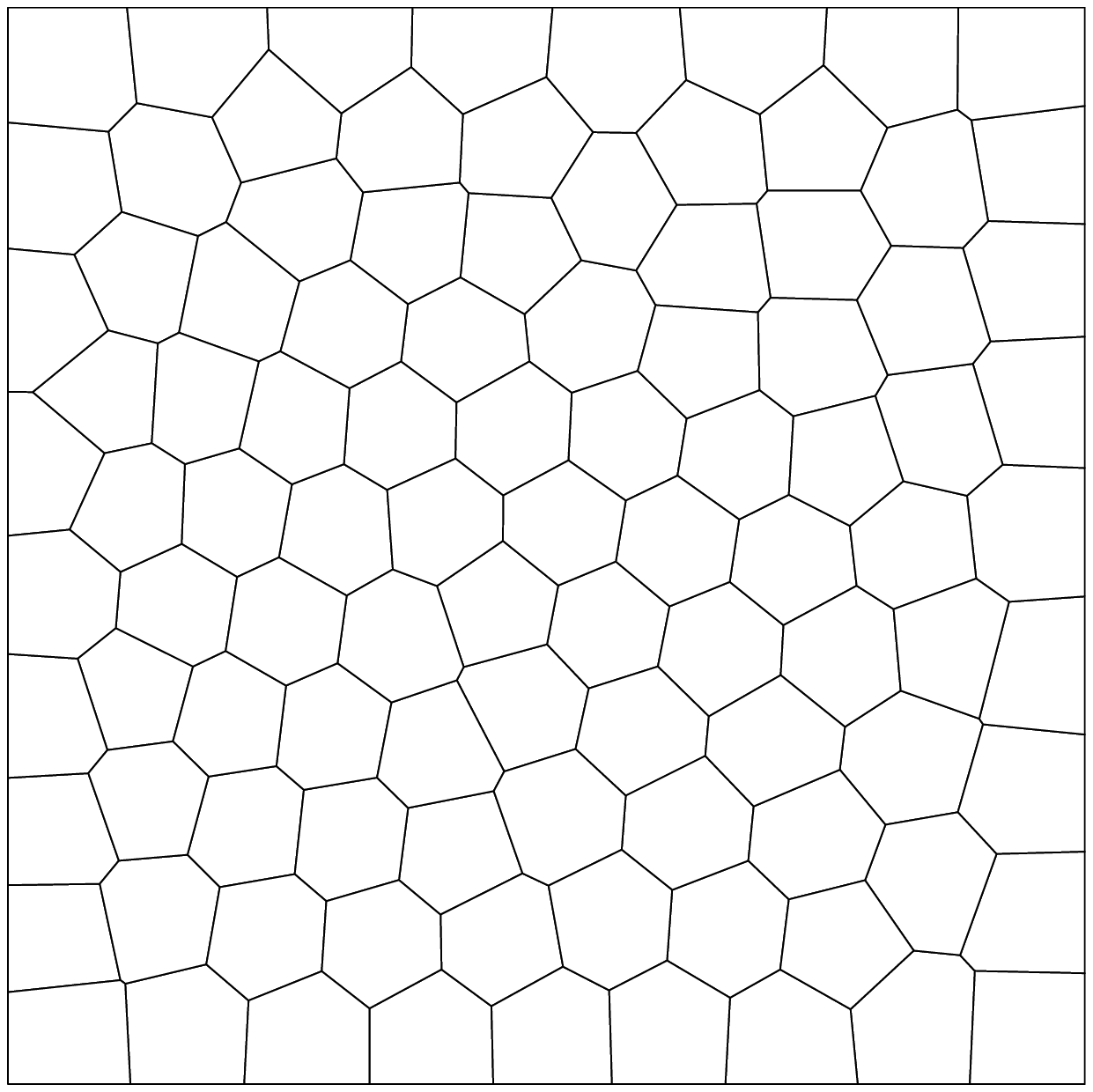}\\
\caption{Examples of Voronoi polygonal meshes (quadrilaterals, pentagons, hexagons) 
with 10 (left) and 100 (right) elements.}
\label{paul_example}
\end{center}
\end{figure}

\begin{figure}[t]
\begin{center}
\subfigure[Quadrilateral mesh of $16384=128\times 128$ elements ($49923$ degrees of freedom).\label{cross_quad}]{
\includegraphics[width=0.32\textwidth]{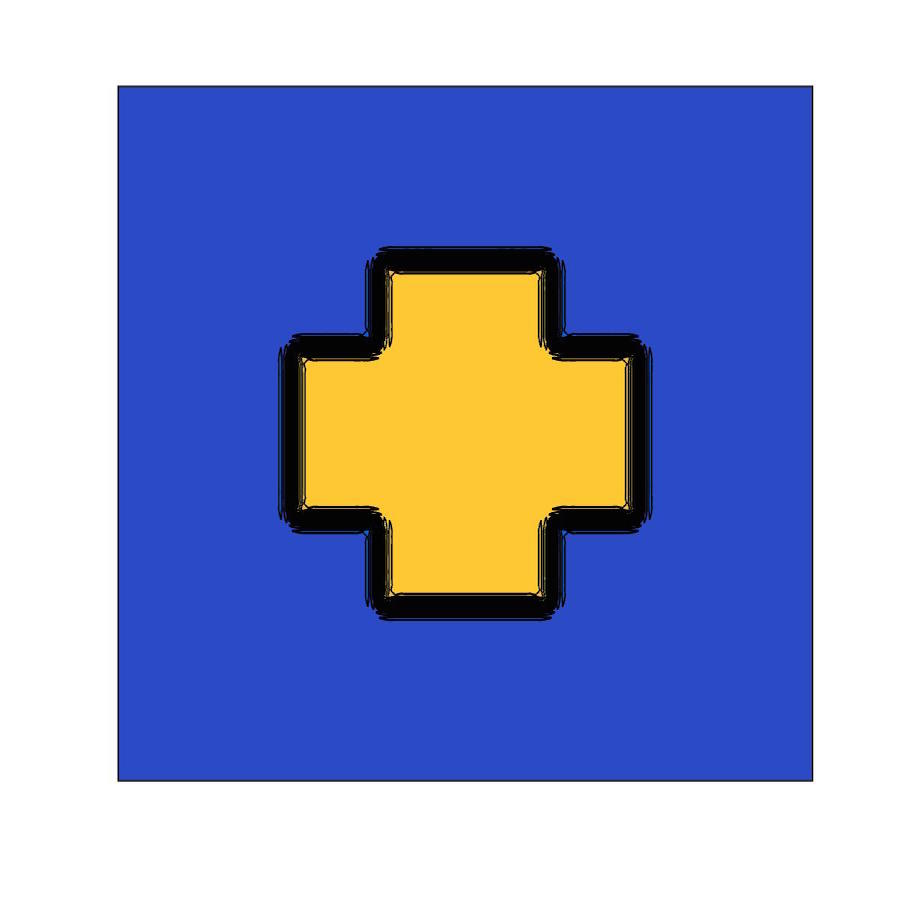}
\includegraphics[width=0.32\textwidth]{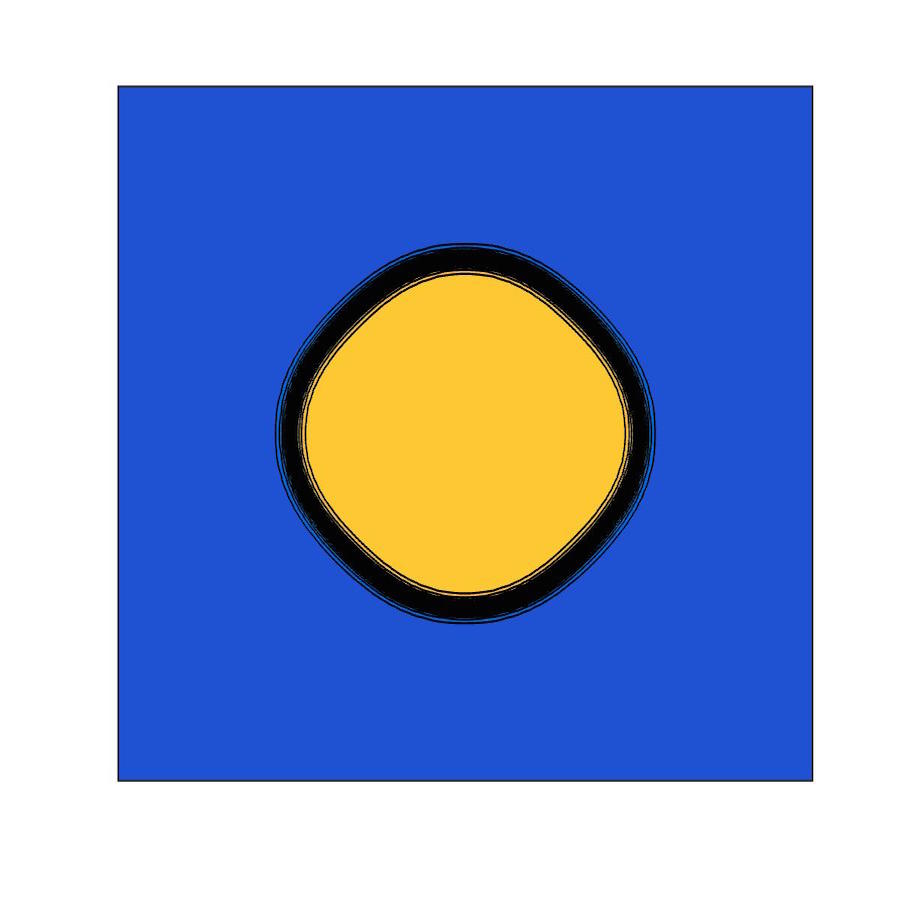}
\includegraphics[width=0.32\textwidth]{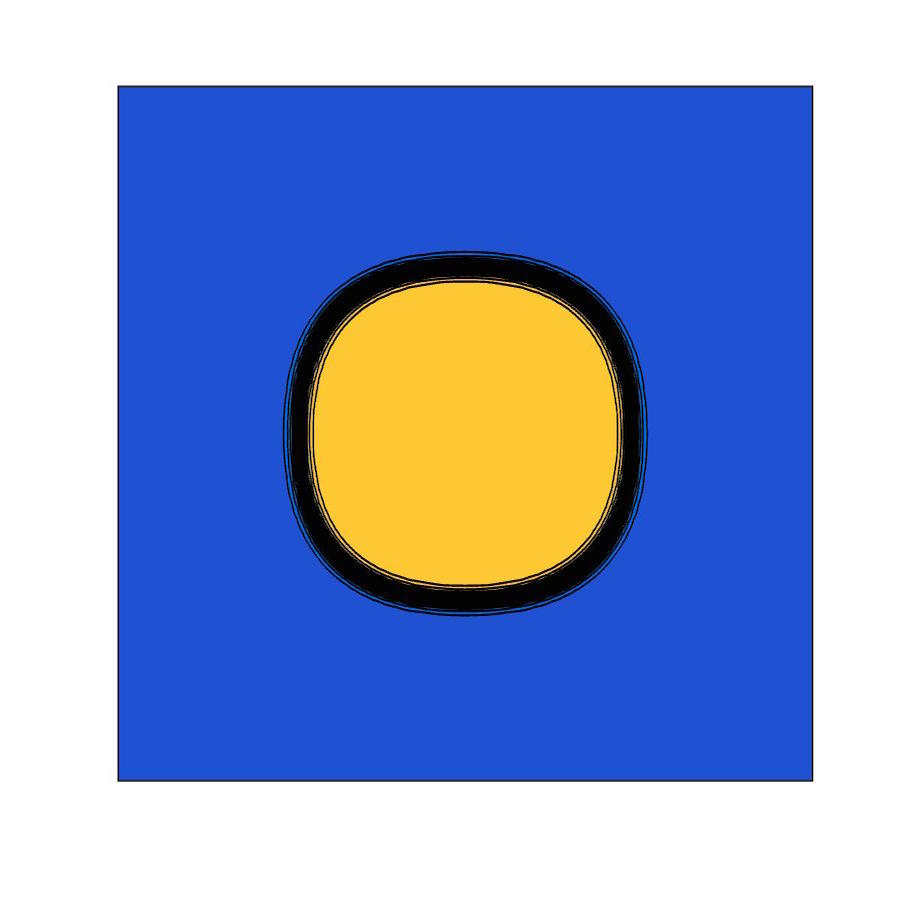}
}
\subfigure[Triangular mesh of $8576$ elements
discretizing the unit square ($13167$ degrees of freedom). \label{cross_tri}]{
\includegraphics[width=0.32\textwidth]{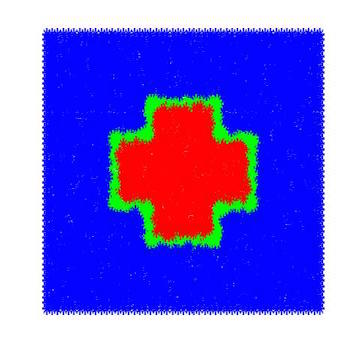}
\includegraphics[width=0.32\textwidth]{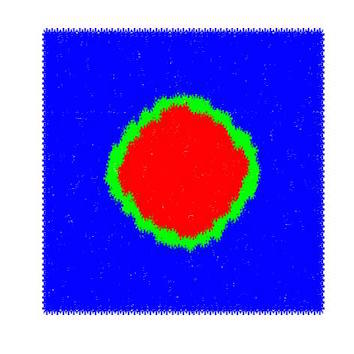}
\includegraphics[width=0.32\textwidth]{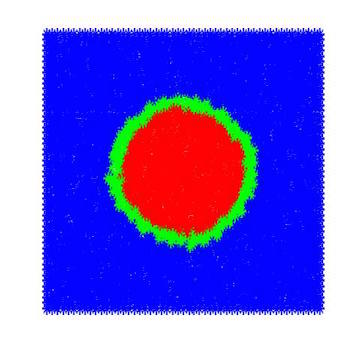}
}
\subfigure[Voronoi polygonal mesh (quadrilaterals, pentagons, hexagons) of $10000$ elements ($59490$ degrees of freedom). \label{cross_paul}]{
\includegraphics[width=0.32\textwidth]{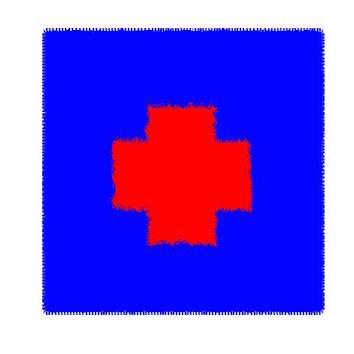}
\includegraphics[width=0.32\textwidth]{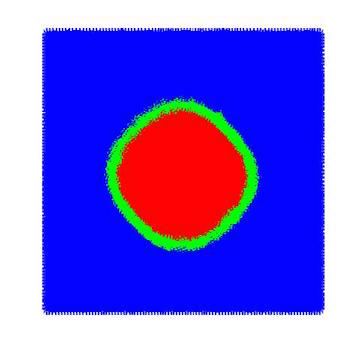}
\includegraphics[width=0.32\textwidth]{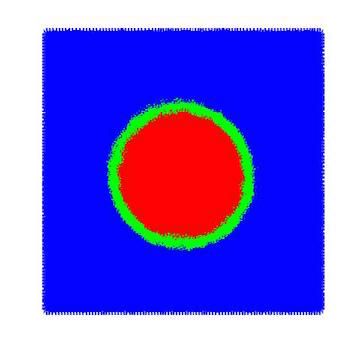}
}
\caption{Test 3: evolution of a cross at three temporal frames  ($t=0, 0.05, 1$). }
\end{center}
\end{figure}

\subsection{Test 4: spinoidal decomposition}
Spinodal decomposition is a physical phenomenon consisting of
the separation of a mixture of two or more components to bulk regions of each.
It occurs when a high-temperature mixture of different components is rapidly cooled. 
To model this separation the initial datum $u_0$ is chosen to be a 
uniformly distributed random perturbation between -1 and 1, 
see Figs.~\ref{spin_quad}, \ref{spin_tri}, \ref{spin_paul} (t=0). The same parameters as in Test 2 are used.
We remark that the three initial random configurations are different.
We consider a quadrilateral mesh with 49923 dof (Fig.~\ref{spin_quad}), 
a triangular mesh with 13167 dof (Fig.~\ref{spin_tri}) and a polygonal mesh with 59590 dof (Fig.~\ref{spin_paul}).
The separation of the two components into bulk regions can be appreciated quite early, 
see Figs.~\ref{spin_quad}, \ref{spin_tri}, \ref{spin_paul} (t=0.01). 
This initial separation happens over a very small time-scale compared to
the motion thereafter. Then, the bulk regions begin to move more slowly, and
separation will continue until the interfaces develop a constant curvature.
In the quadrilateral (Fig.~\ref{spin_quad}) and triangular (Fig.~\ref{spin_tri}) mesh cases,
the final equilibrium configuration is
the square divided into two rectangles, while in the polygonal (Fig.~\ref{spin_paul}) mesh case
the final equilibrium configuration is clearly a circle. The fact that different final configurations are obtained starting from different initial random configurations is consistent with the results in \cite{Gomez-etal:2008}.

\begin{figure}[t]
\begin{center}
\subfigure[Quadrilateral mesh of $16384=128\times 128$ elements ($49923$ degrees of freedom).\label{spin_quad}]{
\includegraphics[width=0.32\textwidth]{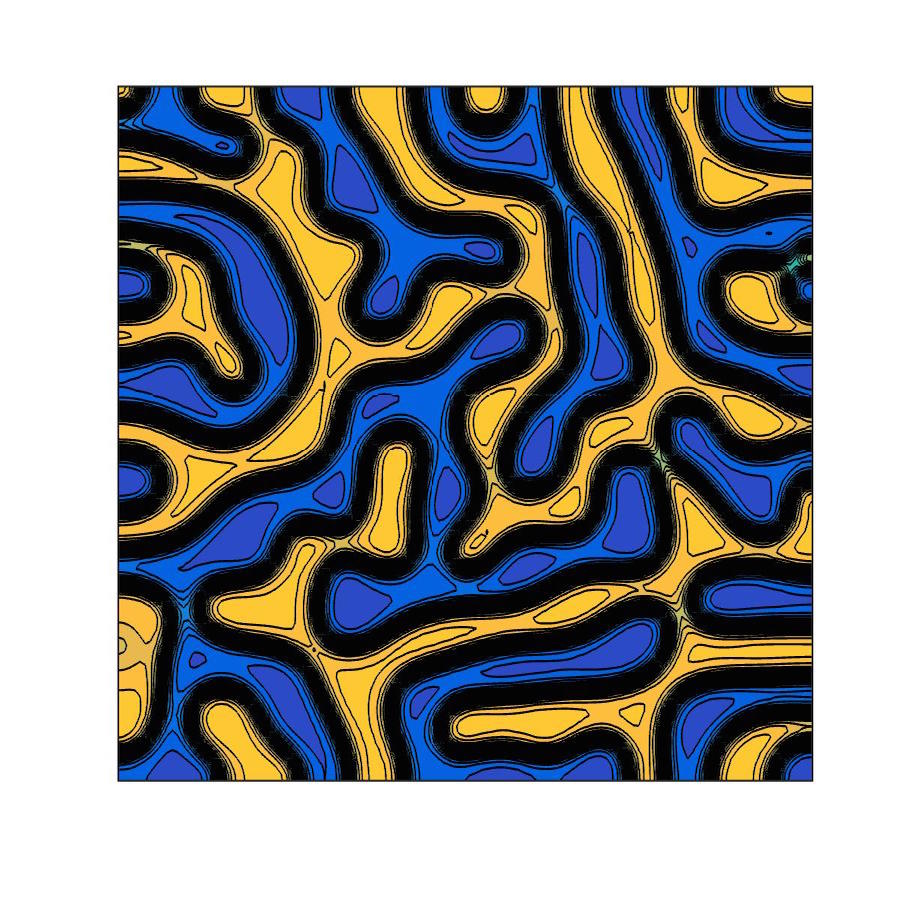}
\includegraphics[width=0.32\textwidth]{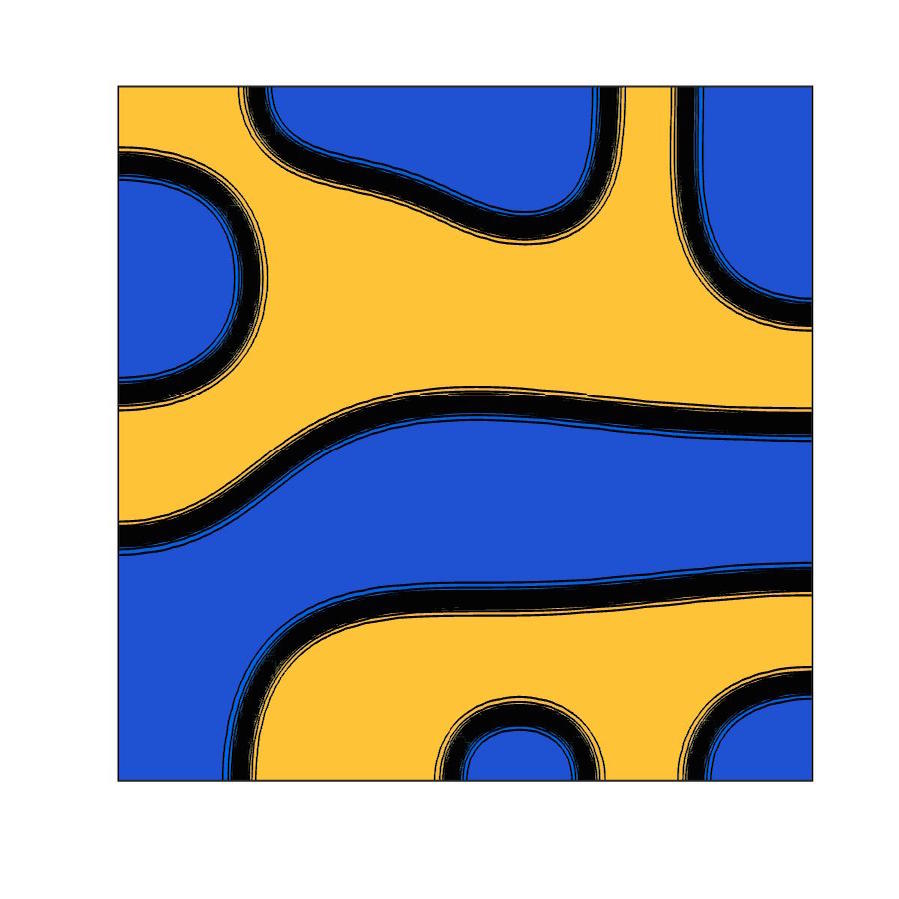}
\includegraphics[width=0.32\textwidth]{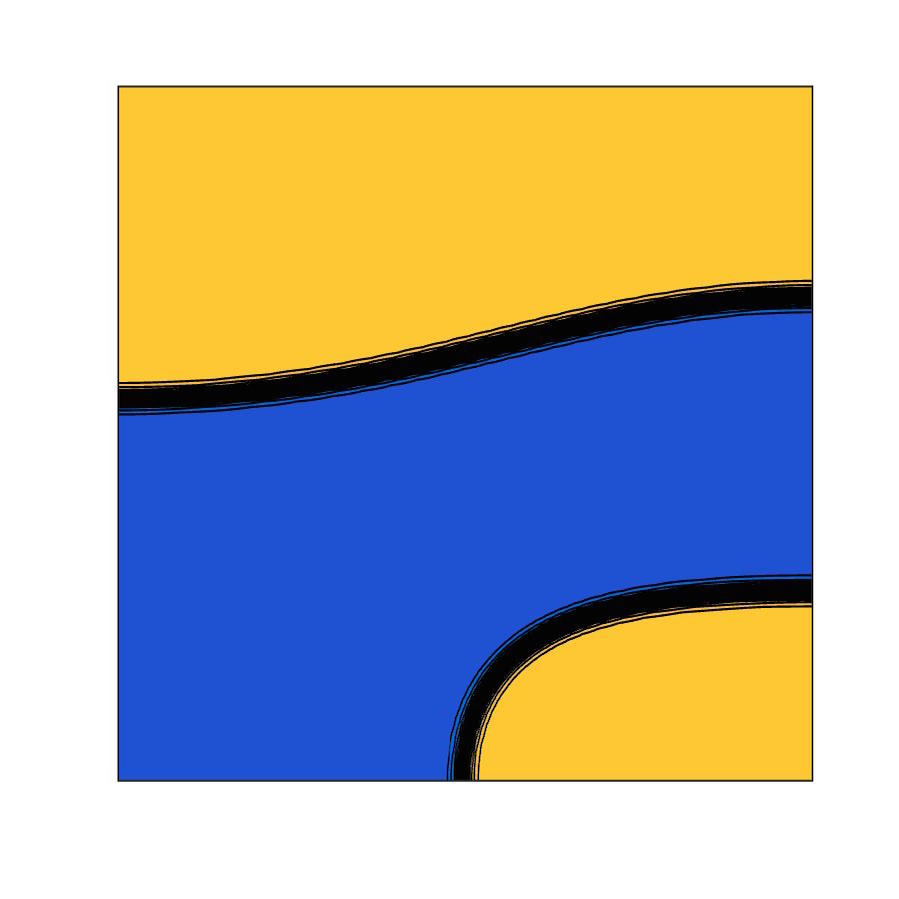}
}
\subfigure[Triangular mesh of $8576$ elements
 ($13167$ degrees of freedom). \label{spin_tri}]{
\includegraphics[width=0.32\textwidth]{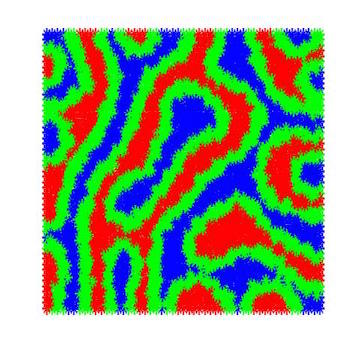}
\includegraphics[width=0.32\textwidth]{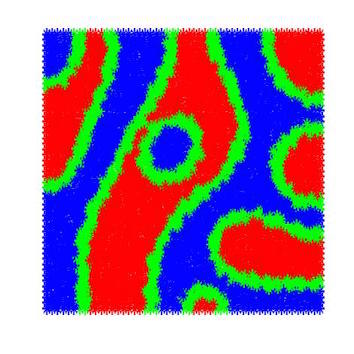}
\includegraphics[width=0.32\textwidth]{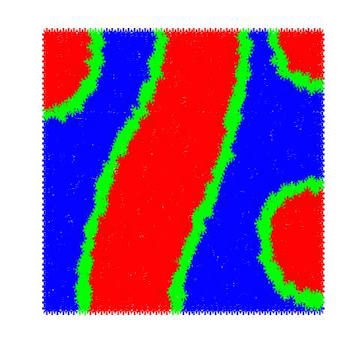}
}
\subfigure[Voronoi polygonal mesh (quadrilaterals, pentagons, hexagons) of $10000$ elements ($59490$ degrees of freedom). \label{spin_paul}]{
\includegraphics[width=0.32\textwidth]{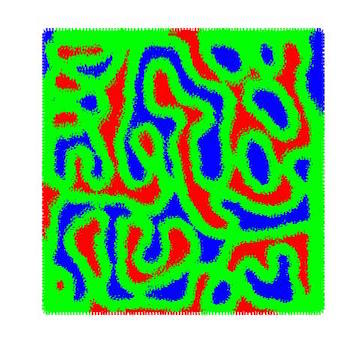}
\includegraphics[width=0.32\textwidth]{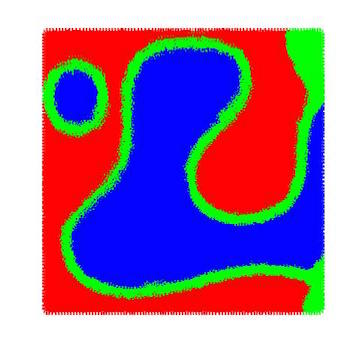}
\includegraphics[width=0.32\textwidth]{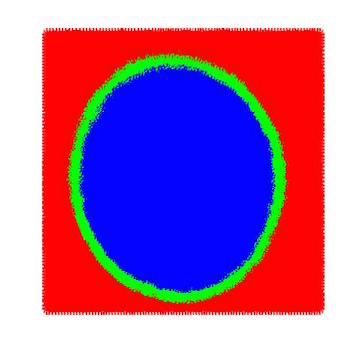}
}
\caption{Test 4: spinoidal decomposition at three temporal frames  
($t=0.01, 0.05, 5$ for the quadrilateral and Voronoi polygonal meshes,
$t=0.075, 0.25, 1.25$ for the triangular mesh). 
}
\end{center}
\end{figure}


\FloatBarrier


\begin{thebibliography}{10}

\bibitem{Adams}
{\sc R.~A. Adams and J.~J.~F. Fournier}, {\it Sobolev spaces}, vol.~140 of Pure
  and Applied Mathematics (Amsterdam), Elsevier/Academic Press, Amsterdam,
  second~ed., 2003.

\bibitem{projectors}
{\sc B.~Ahmad, A.~Alsaedi, F.~Brezzi, L.~D. Marini, and A.~Russo}, {\it
  Equivalent projectors for virtual element methods}, Comput. Math. Appl., 66
  (2013), pp.~376--391.

\bibitem{ABMV14}
{\sc P.~F. Antonietti, L.~Beir{\~a}o~da Veiga, D.~Mora, and M.~Verani}, {\it A
  stream virtual element formulation of the {S}tokes problem on polygonal
  meshes}, SIAM J. Numer. Anal., 52 (2014), pp.~386--404.

\bibitem{volley}
{\sc L.~Beir{\~a}o~da Veiga, F.~Brezzi, A.~Cangiani, G.~Manzini, L.~D. Marini,
  and A.~Russo}, {\it Basic principles of virtual element methods}, Math.
  Models Methods Appl. Sci., 23 (2013), pp.~199--214.

\bibitem{VEM-elasticity}
{\sc L.~Beir{\~a}o~da Veiga, F.~Brezzi, and L.~D. Marini}, {\it Virtual
  elements for linear elasticity problems}, SIAM J. Numer. Anal., 51 (2013),
  pp.~794--812.

\bibitem{hitchhikers}
{\sc L.~Beir{\~a}o~da Veiga, F.~Brezzi, L.~D. Marini, and A.~Russo}, {\it The
  hitchhiker's guide to the virtual element method}, Math. Models Methods Appl.
  Sci., 24 (2014), pp.~1541--1573.

\bibitem{arbitrary}
{\sc L.~Beir{\~a}o~da Veiga and G.~Manzini}, {\it A virtual element method with
  arbitrary regularity}, IMA J. Numer. Anal., 34 (2014), pp.~759--781.

\bibitem{Berrone-VEM}
{\sc M.~F. Benedetto, S.~Berrone, S.~Pieraccini, and S.~Scial{\`o}}, {\it The
  virtual element method for discrete fracture network simulations}, Comput.
  Methods Appl. Mech. Engrg., 280 (2014), pp.~135--156.

\bibitem{Bertozzi:2007b}
{\sc A.~Bertozzi, S.~Esedo{\=g}lu, and A.~Gillette}, {\it Analysis of a
  two-scale {C}ahn-{H}illiard model for binary image inpainting}, Multiscale
  Model. Simul., 6 (2007), pp.~913--936.

\bibitem{Bishop}
{\sc J.~E. Bishop}, {\it A displacement-based finite element formulation for
  general polyhedra using harmonic shape functions}, Internat. J. Numer.
  Methods Engrg., 97 (2014), pp.~1--31.

\bibitem{Brezzi-Falk-Marini}
{\sc F.~Brezzi, R.S. Falk, and L.D. Marini}, {\it Basic principles of mixed
  {V}irtual {E}lement {M}ethods}, Math. Mod. Num. Anal., 48 (2014),
  pp.~1227--1240.

\bibitem{kirchhoff}
{\sc F.~Brezzi and L.~D. Marini}, {\it Virtual element methods for plate
  bending problems}, Comput. Methods Appl. Mech. Engrg., 253 (2013),
  pp.~455--462.

\bibitem{Cahn:1961}
{\sc J.W. Cahn}, {\it On spinodal decomposition}, Acta Metall, 9 (1961),
  pp.~795--801.

\bibitem{hourglass}
{\sc A.~Cangiani, G.~Manzini, A.~Russo, and N.~Sukumar}, {\it Hourglass
  stabilization and the virtual element method}, Internat. J. Numer. Methods
  Engrg.,  (2015).
\newblock to appear.

\bibitem{copolymers}
{\sc R.~Choksi, M.~A. Peletier, and J.~F. Williams}, {\it On the phase diagram
  for microphase separation of diblock copolymers: an approach via a nonlocal
  {C}ahn-{H}illiard functional}, SIAM J. Appl. Math., 69 (2009),
  pp.~1712--1738.

\bibitem{Cockburn-IMU}
{\sc B.~Cockburn}, {\it The hybridizable discontinuous {G}alerkin methods}, in
  Proceedings of the {I}nternational {C}ongress of {M}athematicians. {V}olume
  {IV}, Hindustan Book Agency, New Delhi, 2010, pp.~2749--2775.

\bibitem{population}
{\sc D.~S. Cohen and J.~D. Murray}, {\it A generalized diffusion model for
  growth and dispersal in a population}, J. Math. Biol., 12 (1981),
  pp.~237--249.

\bibitem{DiPietro-Ern-1}
{\sc D.~Di~Pietro and A.~Ern}, {\it A hybrid high-order locking-free method for
  linear elasticity on general meshes}, Comput. Methods Appl. Mech. Engrg., 283
  (2015), pp.~1--21.

\bibitem{Droniou-gradient}
{\sc J.~Droniou, R.~Eymard, T.~Gallou{\"e}t, and R.~Herbin}, {\it Gradient
  schemes: a generic framework for the discretisation of linear, nonlinear and
  nonlocal elliptic and parabolic equations}, Math. Models Methods Appl. Sci.,
  23 (2013), pp.~2395--2432.

\bibitem{DupontScott}
{\sc T.~Dupont and R.~Scott}, {\it Polynomial approximation of functions in
  {S}obolev spaces}, Math. Comp., 34 (1980), pp.~441--463.

\bibitem{Elliott-French:1987}
{\sc C.~M. Elliott and D.~A. French}, {\it Numerical studies of the
  {C}ahn-{H}illiard equation for phase separation}, IMA J. Appl. Math., 38
  (1987), pp.~97--128.

\bibitem{Elliott-French:1989}
\leavevmode\vrule height 2pt depth -1.6pt width 23pt, {\it A nonconforming
  finite-element method for the two-dimensional {C}ahn-{H}illiard equation},
  SIAM J. Numer. Anal., 26 (1989), pp.~884--903.

\bibitem{Elliott-French-Milner:1989}
{\sc C.~M. Elliott, D.~A. French, and F.~A. Milner}, {\it A second order
  splitting method for the {C}ahn-{H}illiard equation}, Numer. Math., 54
  (1989), pp.~575--590.

\bibitem{Elliott-Larsson:1992}
{\sc C.~M. Elliott and S.~Larsson}, {\it Error estimates with smooth and
  nonsmooth data for a finite element method for the {C}ahn-{H}illiard
  equation}, Math. Comp., 58 (1992), pp.~603--630, S33--S36.

\bibitem{Elliott-Zheng:1986}
{\sc C.~M. Elliott and S.~Zheng}, {\it On the {C}ahn-{H}illiard equation},
  Arch. Rational Mech. Anal., 96 (1986), pp.~339--357.

\bibitem{Gillette-2}
{\sc M.~Floater, A.~Gillette, and N.~Sukumar}, {\it Gradient bounds for
  {W}achspress coordinates on polytopes}, SIAM J. Numer. Anal., 52 (2014),
  pp.~515--532.

\bibitem{Paulino-VEM}
{\sc A.~L. Gain, C.~Talischi, and G.~H. Paulino}, {\it On the {V}irtual
  {E}lement {M}ethod for three-dimensional linear elasticity problems on
  arbitrary polyhedral meshes}, Comput. Methods Appl. Mech. Engrg., 282 (2014),
  pp.~132--160.

\bibitem{Gomez-Hughes:2008}
{\sc H.~G{\'o}mez, V.~M. Calo, Y.~Bazilevs, and T.~J.~R. Hughes}, {\it
  Isogeometric analysis of the {C}ahn-{H}illiard phase-field model}, Comput.
  Methods Appl. Mech. Engrg., 197 (2008), pp.~4333--4352.

\bibitem{Gomez-etal:2008}
{\sc H.~G{\'o}mez, V.~M. Calo, Y.Bazilevs, and T.~J.~R.
  Hughes}, {\it Isogeometric analysis of the {C}ahn-{H}illiard phase-field
  model}, Comput. Methods Appl. Mech. Engrg., 197 (2008), pp.~4333--4352.

\bibitem{Cahn-Hilliard:1958}
{\sc J.E.~Hilliard J.W.~Cahn}, {\it Free energy of a non-uniform system. {I}.
  {I}nterfacial free energy}, J. Chem. Phys., 28 (1958), pp.~258--267.

\bibitem{Cahn-Hilliard:1959}
\leavevmode\vrule height 2pt depth -1.6pt width 23pt, {\it Free energy of a
  non-uniform system. {III}. {N}ucleation in a two-component incompressible
  fluid}, J. Chem. Phys., 31 (1959), pp.~688--699.

\bibitem{Kay-Styles-Suli:2009}
{\sc D.~Kay, V.~Styles, and E.~S{\"u}li}, {\it Discontinuous {G}alerkin finite
  element approximation of the {C}ahn-{H}illiard equation with convection},
  SIAM J. Numer. Anal., 47 (2009), pp.~2660--2685.

\bibitem{Korteweg}
{\sc D.J. Korteweg}, {\it Sur la forme que prenent les \'equations du
  mouvements des fluides si l'on tient compte des forces capilaires caus\'ees
  par des variations de densit\'e consid\'erables mains continues et sur la
  th\'eorie de la capillarit\'e dans l'hypoth\'ese d'une varation continue de
  la densit\'e}, Arch. N\'eerl Sci. Exactes Nat. Ser. II,  (1901).

\bibitem{Landau-Ginzburg:1965}
{\sc V.I.~Ginzburg L.D.~Landau}, {\it On the theory of superconductivity}, in
  Collected Papers, D.~ter~Haar L.D.~Landau, ed., Pergamon Oxford, 1965,
  pp.~626--633.

\bibitem{mussels}
{\sc Q.-X. Liu, A.~Doelman, V.~Rottschafer, M.~de~Jager, P.~M.~J. Herman,
  M.~Rietkerk, and J.~van~de Koppel}, {\it Phase separation explains a new
  class of self-organized spatial patterns in ecological systems}, Proceedings
  of the National Academy of Sciences, 110 (2013), pp.~11905--11910.

\bibitem{VEM19}
{\sc G.~Manzini, A.~Russo, and N.~Sukumar}, {\it New perspectives on polygonal
  and polyhedral finite element methods}, Math. Models Methods Appl. Sci., 24
  (2014), pp.~1665--1699.

\bibitem{VemSteklov}
{\sc D.~Mora, G.~Rivera, and R.~Rodr{\'\i}guez}, {\it A virtual element method
  for the {S}teklov eigenvalue problem}.
\newblock CI2MA Pre-Publicaci{\'o}n 2014-27, 2014.

\bibitem{Temam:1989}
{\sc B.~Nicolaenko, B.~Scheurer, and R.~Temam}, {\it Some global dynamical
  properties of a class of pattern formation equations}, Comm. Partial
  Differential Equations, 14 (1989), pp.~245--297.

\bibitem{Oden-tumor:2010}
{\sc J.~T. Oden, A.~Hawkins, and S.~Prudhomme}, {\it General diffuse-interface
  theories and an approach to predictive tumor growth modeling}, Math. Models
  Methods Appl. Sci., 20 (2010), pp.~477--517.

\bibitem{TPPM10}
{\sc C.~Talischi, G.~H. Paulino, A.~Pereira, and I.~F.~M. Menezes}, {\it
  Polygonal finite elements for topology optimization: A unifying paradigm},
  Internat. J. Numer. Methods Engrg., 82 (2010), pp.~671--698.

\bibitem{Thomee-book}
{\sc V.~Thom{\'e}e}, {\it Galerkin finite element methods for parabolic
  problems}, vol.~25 of Springer Series in Computational Mathematics,
  Springer-Verlag, Berlin, second~ed., 2006.

\bibitem{saturn}
{\sc S.~Tremaine}, {\it On the origin of irregular structure in saturn's
  rings}, Astron. J., 125 (2003), pp.~894--901.

\bibitem{Waals:1979}
{\sc {J}.~{D}. van~der {W}aals}, {\it {T}he thermodynamic theory of capillarity
  under the hypothesis of a continuous variation of density}, J. Statist.
  Phys., 20 (1979), pp.~197--244.

\bibitem{Wachspress11}
{\sc E.~Wachspress}, {\it Rational bases for convex polyhedra}, Comput. Math.
  Appl., 59 (2010), pp.~1953--1956.

\bibitem{Wang-1}
{\sc J.~Wang and X.~Ye}, {\it A weak {G}alerkin finite element method for
  second-order elliptic problems}, J. Comput. Appl. Math., 241 (2013),
  pp.~103--115.

\bibitem{Wells-Kuhl-Garikipati:2006}
{\sc G.~N. Wells, E.~Kuhl, and K.~Garikipati}, {\it A discontinuous {G}alerkin
  method for the {C}ahn-{H}illiard equation}, J. Comput. Phys., 218 (2006),
  pp.~860--877.

\bibitem{Wise-Lowengrub:2008}
{\sc S.~M. Wise, J.~S. Lowengrub, H.~B. Frieboes, and V.~Cristini}, {\it
  Three-dimensional multispecies nonlinear tumor growth---{I}: {M}odel and
  numerical method}, J. Theoret. Biol., 253 (2008), pp.~524--543.

\end{thebibliography}

\end{document}